\documentclass[12pt]{amsart}
\usepackage{amsmath,amssymb,latexsym, amsthm, amscd, mathrsfs, stmaryrd}
\usepackage[linktocpage=true]{hyperref}
\usepackage{color}
\usepackage[all]{xy}

\newtheorem{thm}{Theorem} [section]
\theoremstyle{definition}

\theoremstyle{plain}
\newtheorem{prop}[thm]{Proposition}
\newtheorem{lem}[thm]{Lemma}

\newtheorem{conj}[thm]{Conjecture}

\numberwithin{equation}{section}

\theoremstyle{remark}
\newtheorem{rem}{Remark}[section]

\newcommand{\Hom}{\mathrm{Hom}}

\newcommand{\Bl}{\mathcal B}
\newcommand{\C}{\mathbb C}
\newcommand{\D}{D(2|1;\ka)}

\newcommand{\ep}{\epsilon}
\newcommand{\E}{\mathcal E}
\newcommand{\hf}{{\Small \frac12}}

\newcommand{\g}{\mathfrak{g}}
\newcommand{\gl}{\mathfrak{gl}}
\newcommand{\h}{\mathfrak{h}}
\newcommand{\ka}{\zeta}

\newcommand{\la}{\lambda}
\newcommand{\n}{\mathfrak n}
\newcommand{\N}{\mathbb N}
\newcommand{\one}{{\ov 1}}
\newcommand{\osp}{\text{osp}}
\newcommand{\ov}{\overline}
\newcommand{\OO}{\mathcal O}

\newcommand{\Q}{\mathbb Q}

\newcommand{\sll}{\mathfrak{sl}_2}

\newcommand{\Wt}{\text{WT}}
\newcommand{\Z}{\mathbb Z}
\newcommand{\oo}{{\ov 0}}
\newcommand{\hdel}{{h_{2\delta}}}
\newcommand{\hone}{h_{2\ep_1}}
\newcommand{\htwo}{h_{2\ep_2}}

\newcommand{\LL}[1]{L_{{#1}}}
\newcommand{\M}[1]{M_{#1}}
\newcommand{\PP}[1]{P_{{#1}}}
\newcommand{\TT}[1]{T_{{#1}}}

\title[Character formulae in super category $\OO$ for $D(2|1;\ka)$]{Character formulae in category $\OO$ for exceptional Lie superalgebras  $D(2|1;\ka)$ }

\author[S.-J. Cheng]{Shun-Jen Cheng}
\address{Institute of Mathematics, Academia Sinica, Taipei,
Taiwan 10617} \email{chengsj@math.sinica.edu.tw}

\author[W. Wang]{Weiqiang Wang}
\address{Department of Mathematics, University of Virginia, Charlottesville, VA 22904}
\email{ww9c@virginia.edu}

\keywords{exceptional Lie superalgebras, character formulae, tilting modules, Verma flags, composition factors.}
\subjclass[2010]{Primary 17B10}

\begin{document}

\maketitle

\begin{abstract}
We establish character formulae for representations of the one-parameter family of simple Lie superalgebras  $D(2|1;\zeta)$. We provide a complete description of the Verma flag multiplicities of the tilting modules and the projective modules in the BGG category $\mathcal O$ of  $D(2|1;\zeta)$-modules of integral weights, for any complex parameter $\zeta$. The composition factors of all Verma modules in $\mathcal O$ are then obtained.
\end{abstract}

\maketitle

\setcounter{tocdepth}{1}
\tableofcontents

 \section{Introduction}

 \subsection{Background}

Finding irreducible character formulae  is a fundamental problem in representation theory.
For the BGG category $\OO$ of semisimple Lie algebras, Kazhdan-Lusztig theory provides a powerful setting and gives a solution to this problem. For simple Lie superalgebras, new approaches are required as the Weyl groups fail to control the linkage in the category $\OO$ completely. In recent years, character formulae in terms of canonical bases have been developed for BGG category $\OO$ of basic Lie superalgebras of type ABCD; see \cite{Br03, CLW15, CFLW, CKW15, BW13, Bao17} and references therein. There has also been progress on character formulae in the category $\OO$ of the queer Lie superalgebras; see \cite{BrD17} for the most recent update. Some other approach for the irreducible characters in the categories of finite-dimensional modules of basic Lie superalgebras (which are not a BGG parabolic category in general) has also been developed, cf. e.g. \cite{GS10}.

 The basic Lie superalgebras also include three exceptional ones: $\D$, $G(3)$ and $F(3|1)$.
 The existence of the exceptional simple Lie superalgebras, including the one-parameter family $D(2|1;\ka)$,  was first suggested by Freund-Kaplansky \cite{FK}. The finite-dimensional simple modules of exceptional Lie superalgebras were classified by Kac \cite{Kac78}, and they have been studied by various authors; see, for example, \cite{VdJ85, Zou94, Ger00, Ma14, SZ16}.
 However, there has been little progress on the irreducible character formula problem in the BGG category $\OO$ for exceptional Lie superalgebras.

 Recall $\D =\g_\oo \oplus\g_\one$ is a family of simple Lie superalgebras
 of dimension $17$, depending on a  parameter $\ka \in\C\setminus\{0,-1\}$. Its even subalgebra is
 $\g_\oo \cong \sll \oplus \sll \oplus \sll$, and, as an adjoint $\g_{\bar 0}$-module, $\g_\one \cong \C^2\boxtimes \C^2 \boxtimes \C^2$.

 \subsection{The goal}

In this paper, we completely solve the irreducible character formula problem in the BGG category $\OO$ of $\D$-modules of integer weights, for any parameter $\ka$. We provide explicit formulae for Verma flag multiplicities of all tilting modules and projective modules in $\OO$. We then obtain explicit formulae for the composition factors of all Verma modules in $\OO$; this then determines the irreducible characters in $\OO$ as well.

Note $D(2|1;1) \cong \text{osp}(4|2)$, and there is a conceptual approach, due to Bao \cite{Bao17}, to obtain character formula for $\text{osp}(2m|2n)$ via canonical bases arising from quantum symmetric pairs. We also note that the character formulae for  the irreducible {\em finite-dimensional} modules of $\D$ were obtained in \cite{VdJ85} (also cf.~\cite{SZ16}).

 \subsection{The main results and strategy}

The blocks in the BGG category $\OO$ for $\D$ are divided into typical and atypical ones, where a typical block is characterized by having finitely many simple objects.  The characters in typical blocks, which are controlled by the Weyl group $W\cong \Z_2 \times \Z_2 \times \Z_2$, follow from more general work of Gorelik \cite{Gor02} and they can also be obtained directly. We shall focus on the much more challenging atypical blocks, where the results are organized depending on whether the parameter  $\ka \not \in \Q$ or $\ka\in \Q \setminus\{0,-1\}$.

 For $\ka \not \in \Q$, we show that there is exactly one atypical block in $\OO$, i.e., the principal block $\Bl_0$, and we classify the simples in $\Bl_0$. For each $\ka \in \Q\setminus\{0,-1\}$  (which can and will be assumed to be positive without loss of generality), we show that the atypical blocks in $\OO$, denoted by $\Bl_k$, are parametrized by nonnegative integers $k$, where $\Bl_0$ still denotes the principal block. We also classify the simple objects in each block $\Bl_k$. In the category of {\em finite-dimensional} $\D$-modules, the blocks and the simples in each block were classified by Germoni \cite{Ger00}.

The existence of tilting modules in a super category $\OO$ was established by Brundan \cite{Br04}, who in turn superized the idea of Soergel \cite{Soe98} for Kac-Moody algebras. Recall the standard fact that a translation functor sends a tilting module to a direct sum of tilting modules. (A tilting module is understood to be indecomposable in this paper.)

Our construction of tilting modules is inductive in nature, and we construct tilting modules following the anti-dominance order within a $W$-orbit of $8$ (or $4$) weights.  To obtain the character formula for an (arbitrary) tilting module $\TT{\la}$ of highest weight $\la-\rho$ in $\OO$, we proceed in 3 steps as follows.

\begin{enumerate}
\item
Choose a suitable tilting module $\TT{\mu}$ with known Verma flag, which is acted on by a suitable translation functor $\E$.
\item
Determine the Verma flag of the resulting module $\E \TT{\mu}$.

\noindent [Step~(1) is chosen so that a Verma module
$\M{\la}$ of highest weight $\la -\rho$ appears as a submodule in $\E \TT{\mu}$ with multiplicity one and all other Verma modules appearing in the Verma flag of $\E \TT{\mu}$ are of the form $\M{\nu}$ with $\nu\prec \la$ in the Bruhat ordering.]

\item
Verify the module $\E \TT{\mu}$ is indecomposable (with one exception).
\end{enumerate}
It follows from (1)--(3) that $\E \TT{\mu}$ is the desired tilting module $\TT{\la}$.

 Some comments on the 3 steps above are in order. Once Step~ (1) is chosen {\em properly}, it is not difficult to complete Step~ (2).
Almost all translation functors which we shall use arise from tensoring with the $17$-dimensional adjoint module, though there are several exceptions. In the principal block $\Bl_0$ (regardless of $\ka$), there is exactly one singular weight, i.e., the $\rho$-shifted zero weight.
For each block $\Bl_k$ with $k\ge 1$ (for $\ka \in \Q_{>0}$), there are 3 $W$-orbits of singular weights. The character formulae for tilting modules of highest weights near singular weights often require separate treatments as these tilting modules admit irregular Verma flag patterns.

Among other things, Step (3) requires us to know the existence of certain explicit composition factors of Verma modules, so we can apply Soergel duality and BGG reciprocity to show that $\TT{\la}$ has a given Verma module subquotient which shows up in $\E \TT{\mu}$. To that end, we actually establish stronger results by constructing explicitly nonzero homomorphisms between Verma modules. One such homomorphism of independent interest is associated to the non-simple even reflection, and this idea will be pursued elsewhere in greater generality. These Verma module homomorphisms also play a fundamental role in the classification of blocks.

An outcome of our constructions and formulae is that  every tilting module in $\OO$ (with one possible exception) arises from applying a suitable sequence of translation functors to an irreducible Verma module, and the Verma flag multiplicities in an arbitrary tilting module are mostly $1$ and sometimes $2$.
The maximal length of Verma flags of a tilting module is $24$.

Once we know the Verma flag structures of tilting modules, we can determine formulae for the Verma flags for projectives and the composition factors of Verma modules, using Soergel duality \cite{Soe98, Br04} and BGG reciprocity, which then solves the irreducible character problem in $\mathcal O$.

As a consequence of our work we also obtain a classification of projective tilting (=projective injective) modules in $\OO$, for any $\ka$.

 \subsection{Some future work}
This is the first of a series of papers on category $\OO$ for exceptional Lie superalgebras.
 The technique developed in this paper is applicable to other exceptional Lie superalgebras; cf. \cite{CW18} for $G(3)$.
 Our method can also be applied to blocks of non-integral weights or parabolic BGG categories.
 Ultimately, the success of our computational approach relies on the fact that the exceptional Lie superalgebras have small ranks, and  the strategy here should also be applicable to other Lie superalgebras of small ranks to produce explicit character formulae.

In a different direction, we expect that some of the ideas introduced in this paper will lead to a substantial super generalization of the classical theorem of Verma on homomorphisms between Verma modules; cf. Remark~\ref{rem:singular}.

Our work led us to conjecture that the principal blocks of $\D$ for all $\ka$ are equivalent; see Conjecture~ \ref{block0:same}.

While our approach in this paper solves the character formulae problem in $\OO$, it will be very interesting to upgrade it to a super Kazhdan-Lusztig (KL) theory. The super KL theory formulated via $\imath$-canonical bases in \cite{Bao17} in the special case of $\text{osp}(4|2)$ could be helpful in light of Conjecture~ \ref{block0:same}. It will also be interesting to describe the  endomorphism algebra of a projective generator via generators and relations, and this may provide a Koszul $\Z$-grading for $\OO$. One can also try to determine the Jantzen filtration of Verma modules and the Andersen filtration of tilting modules.  

 \subsection{Organization}
The paper is organized as follows.
In Section~\ref{sec:homom}, we review some basic facts of $\D$. We then construct two types of nonzero homomorphisms between Verma modules, one associated to odd reflections and the other associated to the even non-simple reflection.

 In Section~\ref{sec:irrational}, we consider $\D$ with $\ka \not\in \Q$.
We show that the principal block $\Bl_0$ is the only atypical block in $\OO$, and classify the simples in $\Bl_0$.
We apply suitable translation functors to establish the character formulae of tilting modules. The cases for $\TT{\la}$ with $\la$ near the singular weight are treated separately, where one particular tilting module requires extra work.
 Then we determine formulae for the Verma flags for projectives and the composition factors of Verma modules.

 In Section~\ref{sec:rational},  we treat $\D$ with $\ka =p/d \in \Q_{>0}$, with $p,d$ relatively prime positive integers.
We basically repeat the steps as in Section~\ref{sec:irrational}. However, here the details are more involved for the atypical blocks $\Bl_k$, for $k\ge 1$, as there are more singular weights in such a block. The block $\Bl_1$ with $d=1$ is postponed to the next sections because of some extra complications.

   In Sections~\ref{sec:block1} and ~\ref{sec:block:ka=1}, we determine the Verma flags of tilting modules and projective modules in the blocks $\Bl_1$ with $\ka \in \Z_{>0}$, i.e., $d=1$. 
In this case,  some singular weights in $\Bl_1$ come next to each other in the Bruhat ordering, and this makes  the Verma flag structures of tilting modules around singular weights more delicate. The formulations of Verma flag formulae are separated into two cases: $\ka \ge 2$, and $\ka=1$ (i.e., $p=d=1$),
   and they are presented in  Sections~\ref{sec:block1} and ~\ref{sec:block:ka=1}, respectively.
For constructions of some particular tilting modules,
we use translation functors arising from tensoring with a simple $(4p+2)$-dimensional module,
 which reduces to the $6$-dimensional natural module of $\text{osp}(4|2)$ in the case $\ka=1$.

 \vspace{.4cm}
 {\bf Acknowledgment.}
The first author is partially supported by a MoST and an Academia Sinica Investigator grant, while the second author is partially supported by an NSF grant.
We thank University of Virginia and Academia Sinica for hospitality and support since 2014, when this project was initiated.


 %
 %
 \section{Homomorphisms between Verma modules}
 \label{sec:homom}

In this section we work with $\D$ for any $\ka\in\C\setminus\{0,-1\}$. We recall the BGG duality and Soergel duality in category $\OO$. We also review the Harish-Chandra homomorphism for $\D$ following Sergeev.
We then construct homomorphisms between Verma modules associated to even and odd reflections.

\subsection{Preliminaries on $\D$}

Let $\ka\in\C\setminus\{0,-1\}$. The Lie superalgebra $\D$ is the simple Lie superalgebra of basic type associated with the following Cartan matrix \cite[Section 2.5.2]{Kac77}:
\begin{align}\label{Cartan:matrix}
\begin{pmatrix}
  0 & 1 & \ka \\
  -1 & 2 & 0 \\
  -1 & 0 & 2
\end{pmatrix}.
\end{align}
Note that the simple root corresponding to the zero diagonal entry in \eqref{Cartan:matrix} is odd, while the others are even.
There are  isomorphisms of Lie superalgebras with different parameters (cf. \cite[Chapter 1]{CW12})
\begin{equation}
\label{D:iso}
D(2|1;\ka)  \cong
D(2|1; -1-\ka^{-1})  \cong
D(2|1;\ka^{-1}).
\end{equation}

Let $\h$ be a vector space with basis $\{\hdel,\hone,\htwo\}$ and dual basis $\{\delta, \ep_1,\ep_2\}$. We equip the dual $\h^*$ with a bilinear form $(\cdot,\cdot)$ such that $\{\delta, \ep_1,\ep_2\}$ are orthogonal and
\begin{align*}
(\delta, \delta) = -(1+\ka),
\quad
(\ep_1, \ep_1) = 1,
\quad
(\ep_2, \ep_2) = \ka.
\end{align*}
The simple coroots in the Cartan subalgebra $\h$ and the corresponding simple roots in $\h^*$ of $\D$ associated with the Cartan matrix \eqref{Cartan:matrix} are realized respectively as
\begin{align*}
\Pi^\vee&=\{\alpha_0^\vee=\frac{1+\ka}{2}\hdel+\hf\hone+\frac{\ka}{2}\htwo,\alpha_1^\vee=\hone,\alpha_2^\vee=\htwo\},\\
\Pi&=\{\alpha_0=\delta-\ep_1-\ep_2,\alpha_1=2\ep_1,\alpha_2=2\ep_2\}.
\end{align*}
The Dynkin diagram associated to $\Pi$ is depicted as follows:
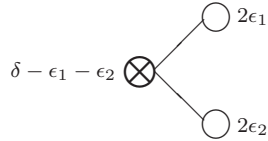
\begin{center}
\setlength{\unitlength}{0.16in}
\begin{picture}(4,6)
\put(4,1.3){\makebox(0,0)[c]{$\bigcirc$}}
\put(4,4.8){\makebox(0,0)[c]{$\bigcirc$}}
\put(1.5,3){\makebox(0,0)[c]{$\bigotimes$}}
\put(3.6,1.4){\line(-1,1){1.6}}
\put(3.6,4.7){\line(-1,-1){1.6}}
\put(5.2,4.8){\makebox(0,0)[c]{\tiny $2\ep_1$}}
\put(5.2,1.2){\makebox(0,0)[c]{\tiny $2\ep_2$}}
\put(-1,3){\makebox(0,0)[c]{\tiny $\delta-\ep_1-\ep_2$}}
\end{picture}
\end{center}

Let $\Phi$, $\Phi_{\bar 0}$ and $\Phi_{\bar 1}$ stand for the sets of roots, even roots and odd roots, respectively. Let $\Phi^+$ further denote the set of positive roots with respect to $\Pi$ and set $\Phi^+_i=\Phi^+\cap\Phi_i$, for $i=\bar{0},\bar{1}$. We have
\begin{align*}
\Phi^+_{\bar 0}=\{2\delta,2\ep_1,2\ep_2\},\quad\Phi^+_{\bar 1}=\{\delta-\ep_1-\ep_2,\delta+\ep_1-\ep_2,\delta-\ep_1+\ep_2,\delta+\ep_1+\ep_2\}.
\end{align*}
Thus, we have $\g_\oo \cong \sll \oplus \sll \oplus \sll$ and, as a $\g_{\bar 0}$-module, we have $\g_\one \cong \C^2\boxtimes \C^2 \boxtimes \C^2$,
where $\C^2$ is the natural  representation of $\sll$. Let $$X=\Z\delta+\Z\ep_1+\Z\ep_2$$ be the weight lattice of $\g$.

Denote by $\{e_i,f_i,h_i\}$ the Chevalley generators corresponding to $\alpha_i$, $i=0,1,2$. We shall also write $e_{\alpha_i}=e_i$ and $f_{\alpha_i}=f_i$, for $i=0,1,2$, in the sequel. We define the following positive and negative root vectors:
\begin{align*}
&e_{\delta+\ep_1-\ep_2}:=[e_0,e_1],\quad f_{\delta+\ep_1-\ep_2}:=[f_0,f_1],\\
&e_{\delta-\ep_1+\ep_2}:=[e_2,e_0],\quad f_{\delta-\ep_1+\ep_2}:=[f_0,f_2],\\
&e_{\delta+\ep_1+\ep_2}:=[[e_0,e_1],e_2],\quad f_{\delta+\ep_1+\ep_2}:=[[f_1,f_0],f_2],\\
&e_{2\delta}=\left(\frac{1}{1+\ka}\right)^2[[e_0,e_1],[e_0,e_2]],\quad f_{2\delta}=[[f_0,f_1],[f_0,f_2]].
\end{align*}
Then these vectors, together with the nine Chevalley generators, form a basis for $\D$.

The following commutation relations will be useful later on:
\begin{align}\label{eq:comm:rel}
\begin{split}
&[e_{\delta\pm\ep_1\pm\ep_2},f_{\delta\pm\ep_1\pm\ep_2}]=\frac{1+\ka}{2}\hdel\mp\hf\hone\mp\frac{\ka}{2}\htwo,\\
&[e_{2\delta},f_{2\delta}]=\hdel.
\end{split}
\end{align}

\subsection{BGG category $\OO$}

Denote by $\g_\alpha$ the root subspace for $\alpha\in\Phi$. Let $\n_\pm=\bigoplus_{\alpha\in\Phi^\pm}\g_\alpha$ so that we have $\g=\n_-\oplus\h\oplus\n_+$. Let
\[
\rho=\rho_{\bar 0}-\rho_{\bar 1},
\quad \text{ where }\rho_{\bar 0}=\hf\sum_{\alpha\in\Phi^+_{\bar 0}} \alpha,\quad
\rho_{\bar 1} =\hf\sum_{\beta\in\Phi^+_{\bar 1}}\beta.
\]
Then we have
\begin{align*}
\rho =-\delta+\ep_1+\ep_2 \; (=-\alpha_0).
\end{align*}

A weight $\la \in \h^*$ is called {\em atypical} if
\begin{equation}
 \label{eq:atypicality}
(\la+\rho, \alpha) =0, \quad \text{ for some } \alpha \in \Phi_{\bar 1};
\end{equation}
 otherwise it is called {\em typical}.
We often use the isomorphism
\begin{equation}
  \label{eq:lattice}
X \stackrel{\cong}{\longrightarrow}  \Z^3,
\qquad
\la \mapsto f_{\la+\rho},
\end{equation}
where $ f_{\la+\rho} =(x,y,z)$, if $\la+\rho =x\delta +y\ep_1 +z\ep_2$;
sometimes we refer to $(x,y,z)$ as a $\rho$-shifted weight.
A $\rho$-shifted weight $(x,y,z)$ is atypical (which corresponds to an atypical $\la$) if and only if it satisfies one of the $4$ equations
$
(\ka+1) x \pm y \pm \ka z =0,
$
or equivalently, if and only if it satisfies one of the $4$ equations
\begin{equation}
 \label{atypical}
 \ka (x\pm z) + (x \pm y)=0.
\end{equation}

For $\la \in \h^*$ denote by $\M{\la}$ the Verma module of highest weight $\la-\rho$. That is, $M_\la= U(\g) \otimes_{U(\h+\n_+)} \C_{\la-\rho}$, where $\C_{\la-\rho}$ is the $1$-dimensional $U(\h+\n_+)$-module with $\h$ acting with the weight $\la-\rho$ and $\n_+$ acting by $0$.
Denote by $v_\la^+$ the highest weight vector in $\M{\la}$ (of highest weight $\la-\rho$).
The unique irreducible quotient module of $\M{\la}$ will be denoted by $\LL{\la}$.

Denote by $\OO$ the BGG category of $\D$-modules of integral weights
with respect to $\g=\n_-\oplus\h\oplus\n_+$. Then $\M{\la}, \LL{\la}$, for $\la\in X$, are objects in $\OO$.
The category $\OO$ has enough projectives.
Denote by $\PP{\la}$ the projective cover of $\LL{\la}$ in $\OO$, for $\la \in X$.
It is well known that the projective module $\PP{\la}$ admits a Verma flag.
Denote by $(V : \M{\mu})$ the multiplicity of $\M{\mu}$
in a (or any) Verma flag of a module $V$ admitting a Verma flag.
We have the BGG reciprocity:
\begin{equation}  \label{BGG}
(\PP{\la} : \M{\mu}) =[\M{\mu}: \LL{\la}], \qquad \text{ for }\la, \mu \in X.
\end{equation}

We denote by $\TT{\la}$ the tilting module in $\OO$ with highest weight $\la-\rho$, for $\la \in X$.
By Brundan \cite{Br04}, which is a super generalization of Soergel \cite{Soe98}, the tilting module $\TT{\la}$ exists.
By definition, a tilting module admits a Verma flag and a dual Verma flag.
A certain equivalence of categories implies the following Soergel duality (cf. \cite{Soe98, Br04}):
\begin{equation}  \label{tilt=proj}
(\PP{\la} : \M{\mu}) = (\TT{-\la} : \M{-\mu}),  \qquad \text{ for }\la, \mu \in X.
\end{equation}
The duality \eqref{tilt=proj} together with BGG reciprocity \eqref{BGG} implies the tilting character formula:
\begin{equation}  \label{tiltingD}
(\TT{-\la} : \M{-\mu}) =[\M{\mu}: \LL{\la}], \qquad \text{ for }\la, \mu \in X.
\end{equation}

The following simple observation will be used several times later on, and so we formulate it explicitly for the convenience of referring.

\begin{lem}
  \label{tilting:comp}
Let $T_{\la} {\longrightarrow} T_{\mu}$ be an epimorphism of tilting modules in $\OO$. Let $\E$ be an exact functor on $\OO$ such that $\E T_\mu=T_{\mu'}$. Then there exists an epimorphism
$\E  T_{\la} {\longrightarrow} T_{\mu'}$.
\end{lem}

\subsection{Central characters of $\g$}


Define a polynomial $P\in\C[\delta,\ep_1,\ep_2]$ by
\begin{align}\label{form:P}
P:=(\delta-\ep_1-\ep_2)(\delta+\ep_1-\ep_2)(\delta-\ep_1+\ep_2)(\delta+\ep_1+\ep_2).
\end{align}
Denote by $\mathcal Z (\g)$ the center of the enveloping algebra $U(\g)$.
According to \cite[\S0.6.6]{Serg}, the image of the Harish-Chandra homomorphism $HC: \mathcal Z (\g) \rightarrow S(\h^*)$ is \begin{align}\label{im:HC}
\text{im} (HC) = \C\Big[ -\frac{\delta^2}{1+\ka}+\ep^2_1+\frac{\ep_2^2}{\ka} \Big] + P\cdot\C[\delta^2,\ep^2_1,\ep^2_2].
\end{align}
Here the element $-\frac{\delta^2}{1+\ka}+\ep^2_1+\frac{\ep_2^2}{\ka}$ is the image of the (suitably normalized) Casimir element under $HC$.
For $\la=a\delta+b\ep_1+c\ep_2 \in \h^*$,  the eigenvalue of the Casimir element on $L_\la$ is given by
\begin{align}\label{action:cas}
&c_\la:=\Big( \la,-\frac{\delta^2}{1+\ka}+\ep^2_1+\frac{\ep_2^2}{\ka} \Big)
=-(1+\ka){a^2}+b^2+\ka{c^2}.
\end{align}

Recall the typical and atypical weights from \eqref{eq:atypicality}. The Weyl group $W$ of the Lie superalgebra $\g$ is understood as the Weyl group of the even subalgebra $\g_{\bar 0}$, i.e.,
\[
W\cong  \Z_2 \times \Z_2 \times \Z_2,
\]
which acts on $\h^*=\C^3$ by sign changes. The following statement is well known, and we give a self-contained proof. (Recall that two weights are linked if the irreducible modules corresponding to these highest weights lie in the same block in $\mathcal O$.)

\begin{lem}\label{lem:central}
Let $\ka\in\C\setminus\{0,-1\}$.
\begin{itemize}
\item[(1)] A typical weight and an atypical weight in $\h^*$ cannot be linked.
\item[(2)] If $f$ and $f'$ are two typical weights with $W f\cap W f'=\emptyset$, then the two sets of weights $\{wf|w\in W\}$ and $\{wg|w\in W\}$ cannot be linked.
\end{itemize}
\end{lem}

\begin{proof}
Let $f$ be typical and $g$ be atypical. Then $(f,P)\not=0$, while $(g,P)=0$. Thus, by \eqref{im:HC}, $f$ and $g$ have different central characters, and hence Part~ (1) follows.

Let $f$ and $f'$ be two typical weights and $W f\cap W f'=\emptyset$. Then either $(f,P)\not=(f',P)$, or else we can find an element $h\in\C[\delta^2,\ep^2_1,\ep^2_2]$ such that $(f,h)\not=(f',h)$. Either way, we can find an element in $P\cdot\C[\delta^2,\ep_1^2,\ep_2^2]$ separating $f$ and $f'$ so that they have different central characters. Part (2) now follows.
\end{proof}

\subsection{Homomorphisms associated to odd reflections}

We determine the singular vectors in a Verma module  associated to odd reflections. This is easy for $\D$ as its positive odd roots have small heights.

\begin{lem}
\label{lem:oddhom}
Assume $\ka\in\C\setminus\{0,-1\}$. Let $\la \in \h^*$ and $\gamma \in \Phi_{\bar 1}^+$ be such that $(\la,\gamma) =0$.
Then there is a nonzero singular vector
in $\M{\la}$ of weight $\la-\rho -\gamma$, and hence
$\Hom \big(\M{\la -\gamma}, \M{\la} \big) \neq 0$.
In particular, we have $[\M{\la}: \LL{\la -\gamma}] > 0$.
\end{lem}

\begin{proof}
Write $\la=a\delta+b\ep_1+c\ep_2$, where $a,b,c\in\Z$. The identity $(\la,\gamma)=0$ puts relations on the integers $a$, $b$, and $c$ depending on $\gamma$. Recall that $v^+_\la$ denotes a highest weight vector in $M_\la$ of highest weight $\la-\rho$.
We shall explicitly write down the formulae for singular vectors case-by-case in (1)--(4) below. Using the identities in \eqref{eq:comm:rel} it is straightforward to verify that these vectors are indeed singular.

(1) If $(\la,\gamma)=0$ and $\gamma=\delta-\ep_1-\ep_2$ is the simple odd root $\alpha_0$, then $f_0 v^+_\la$ is a singular vector of weight $\la-\rho-\gamma$.

(2) Assume that $\gamma=\delta+\ep_1-\ep_2$ and $(\la,\gamma)=0$. Then $$\big(f_0f_1+b[f_0,f_1]\big)v^+_\la$$ is a singular vector of weight $\la-\rho-\gamma$.

(3) Next suppose that $\gamma=\delta-\ep_1+\ep_2$ and $(\la,\gamma)=0$. Then $$\big(f_0f_2+c[f_0,f_2]\big)v^+_\la$$ is a singular vector of weight $\la-\rho-\gamma$.

(4) Finally assume that $\gamma=\delta+\ep_1+\ep_2$ and $(\la,\gamma)=0$. Then
\begin{align*}
\Big(f_0f_1f_2+bf_2[f_0,f_1]+cf_1[f_0,f_2]-\left(b+c+bc\right)[f_1,[f_0,f_2]]\Big)v^+_\la
\end{align*}
is a singular vector of weight $\la-\rho-\gamma$.
\end{proof}

\begin{rem}
One can show that the singular vectors of those given weights in Lemma~\ref{lem:oddhom} are unique up to a scalar multiple, by an elementary albeit somewhat tedious calculation. We do not need the uniqueness result in this paper.
\end{rem}

 \subsection{Homomorphisms associated to even reflections}

For $\gamma\in\Phi^+_{\bar 0}$, denote by $s_\gamma$ the reflection associated with $\gamma$. That is, for $\nu\in\h^*$, we have $s_\gamma\nu=\nu-\langle\nu,h_{\gamma}\rangle\gamma$.
We determine the singular vectors in a Verma module associated to such even reflections, including the most interesting one associated to the even non-simple reflection.
Denote by $\N$ the set of nonnegative integers.

\begin{lem}
\label{lem:evenhom}
For any $\ka\in\C\setminus\{0,-1\}$ let $\la \in \h^*$ and $\gamma \in \Phi_{\bar 0}^+$. Suppose that $\langle\la,h_{\gamma}\rangle =n\in\N$.
Then there is a nonzero singular vector in $\M{\la}$ of weight $\la-\rho-n\gamma$, and hence
$\Hom \big(\M{s_\gamma \la}, \M{\la} \big) \neq 0.$
In particular, we have $[\M{\la}: \LL{s_\gamma \la}] > 0.$
\end{lem}

\begin{proof}
First, recall that $v^+_\la$ is a nonzero highest weight vector in $M_\la$ of highest weight $\la-\rho$.

In the case when $\gamma$ is a simple root (i.e., $\gamma =2\ep_2$ or $2\ep_3$), we have
$
e_\gamma f_{\gamma}^n v_\la^+=0,
$
and hence $f_{\gamma}^n v_\la^+$ is the desired singular vector in $\M{\la}$. 

Now consider the case $\gamma =2\delta$ so that $\la=n\delta+b\ep_1+c\ep_2$, for some $b,c\in\Z$. Note that we have
\begin{align*}
\langle\la-\rho+\rho_{\bar 0},\hdel\rangle=\langle\la,\hdel\rangle+\langle\rho_{\bar 0}-\rho,\hdel\rangle = n+2,
\end{align*}
and thus $e_{2\delta}f_{2 \delta}^{n+2} v_\la^+=0$.
Now set
\begin{equation}  \label{sing:even}
u =e_{\delta-\ep_1-\ep_2}e_{\delta+\ep_1-\ep_2}e_{\delta+\ep_1+\ep_2}e_{\delta-\ep_1+\ep_2} f_{2 \delta}^{n+2} v_\la^+.
\end{equation}
Clearly $u \in \M{\la}$ has weight ${s_\gamma \la}-\rho$.
The lemma follows once we verify that
\begin{enumerate}
\item
$u \neq 0$;
\item
$u$ is a singular vector in $\M{\la}$.
\end{enumerate}

Note that
\begin{align*}
&[e_{2\ep_1}, e_{\delta-\ep_1-\ep_2}] =e_{\delta+\ep_1-\ep_2}, \quad
[e_{2\ep_1}, e_{\delta-\ep_1+\ep_2}] =e_{\delta+\ep_1+\ep_2},\quad
e_{2\ep_1} e_{-2 \delta}^{n+2} v_\la^+ =0,\\
&[e_{2\ep_1}, e_{\delta+\ep_1-\ep_2}] =[e_{2\ep_1}, e_{\delta+\ep_1+\ep_2}]  =[e_{2\ep_1}, f_{2\delta}] =e^2_\nu=0,\quad\nu=\delta\pm\ep_1\pm\ep_2.
\end{align*}
From this we see that $e_{2\ep_1} u=0$.

Note that for any permutation of four letters $\tau$ permuting the four roots $\delta\pm\ep_1\pm\ep_2$,  the vector $e_{\tau(\delta-\ep_1-\ep_2)}e_{\tau(\delta+\ep_1-\ep_2)}e_{\tau(\delta+\ep_1+\ep_2)}e_{\tau(\delta-\ep_1+\ep_2)} f_{2 \delta}^{n+2} v_\la^+$
is equal to the vector in \eqref{sing:even} up to a sign. Now, an almost identical argument gives us $e_{2\ep_2} u=0$ as well.

Clearly we also have $e_{\delta-\ep_1-\ep_2} u=0$. Hence (2) follows.

Part (1) follows by an expansion of $u$ in the form $U(\n_-) v_\la^+$. Indeed, by a tedious direct computation, the vector in \eqref{sing:even} is a nonzero scalar multiple of the following vector:
\begin{align}
  \label{u:singular}
\begin{split}
u=\bigg((\beta\gamma+\xi)&f_{2\delta}^n
-n(1+\ka)\left(\beta f_{\delta+\ep_1-\ep_2}f_{\delta-\ep_1+\ep_2}f_{2\delta}^{n-1} + \eta f_{\delta+\ep_1+\ep_2}f_0f^{n-1}_{2\delta}\right)\\
&+n(1+\ka)\left(f_1f_0f_{\delta-\ep_1+\ep_2}f_{2\delta}^{n-1} +f_2f_0f_{\delta+\ep_1-\ep_2}f_{2\delta}^{n-1}\right)
\\&-(1+\ka)^2n(n-1)f_{\delta+\ep_1+\ep_2}f_0f_{\delta+\ep_1-\ep_2}f_{\delta-\ep_1+\ep_2}f_{2\delta}^{n-2}\bigg)v^+_\la,
\end{split}
\end{align}
where we have used the following notation:
\begin{align*}
&2\beta= (2+n)(1+\ka)+b+\ka c,\quad
2\gamma=(2-n)(1+\ka)-b+\ka c-2,\\
&2\eta= (4-n)(1+\ka)-b+\ka c,\quad
2\xi=(1+n)(1+\ka)+b+\ka c.
\end{align*}
From this formula we see that the vector $u$ is nonzero, since $n\in\N$ and $\ka\not=-1$.
\end{proof}

\begin{rem}
  \label{rem:singular}
The singular vector formula \eqref{sing:even} is inspired by a similar construction in \cite[\S5]{KW95} in the setting of affine superalgebras.
It can lead to more general homomorphisms between Verma modules for basic Lie superalgebras beyond what has been known (cf. \cite[Chapter~9]{Mu12}, under the restrictive Hypothesis~9.2.4). This will be pursued in greater generality in a separate work.
\end{rem}

 \section{Character formulae in $\OO$, for $\ka \not\in \Q$}
  \label{sec:irrational}

In this section we assume  $\ka \not\in \Q$. In this section we obtain the Verma flag multiplicities of every tilting module and projective module in $\OO$. We also describe the composition factors of every Verma module. This gives a solution to the irreducible character problem in $\mathcal O$.

\subsection{Classification of blocks}

Set
\[
\Wt_0 =\{ (\pm n, \pm n, \pm n) \mid n\in \Z_{>0}\} \cup \{0,0,0\}.
\]
The Bruhat ordering $\preceq$ on the  lattice $\Z^3$  is defined to be the partial ordering which is the transitive closure of the following relations:
\begin{enumerate}
\item
$(-a,b,c) \prec (a,b,c)$ if $a>0$,
\item
$(a,-b,c) \prec (a,b,c)$ if $b>0$,
\item
$(a,b,-c) \prec (a,b,c)$ if $c>0$,
\item
$( n, \sigma n, \tau n) \prec ( n+1, \sigma (n+1), \tau (n+1) )$, for any $n\in\Z$ and $\sigma, \tau =\pm$.
\end{enumerate}
This restricts to a Bruhat ordering on $\Wt_0$.
Via the lattice isomorphism \eqref{eq:lattice}, this defines a Bruhat ordering on the weight lattice $X$. We note that the Bruhat ordering thus defined is consistent with the Harish-Chandra homomorphism; cf. \eqref{im:HC}. Furthermore, by Lemmas \ref{lem:oddhom} and \ref{lem:evenhom}, we have $f\prec g$ if and only if there exists a sequence of weights $f=f_1\prec f_2\prec \cdots\prec f_k=g$ such that $\text{Hom}_\g(M_{f_i},M_{f_{i+1}})\not=0$, for all $i=1,\cdots k-1$.

\begin{prop}
  \label{prop:blocks:irr}
Assume that $\ka\not\in\Q$.
\begin{enumerate}
\item
Each typical block is of the form $\Bl_{n,m,\ell}$,
whose simple objects are the simple modules  $\LL{\pm n, \pm m, \pm\ell}$,
for some fixed $n, m,\ell \in \N$, not all equal.

 \item
There is one unique atypical block (i.e., the principal block) $\Bl_0$, whose simple objects are the simple modules of the form $\LL{a,b,c}$, for  $(a,b,c)\in \Wt_0$.
\end{enumerate}
\end{prop}

\begin{proof}
We first determine the atypical integral weights.
Recall the atypicality condition from \eqref{atypical} that a $\rho$-shifted weight $f=(x,y,z)$ is atypical if and only if $\ka (x\pm z) + (x \pm y)=0.$
Thanks to $\ka \not\in \Q$, we conclude that $f=(x,y,z)$ is atypical if and only if $|x| =|y|=|z|$.
Hence $\Wt_0$ is precisely the set of atypical weights.

It follows that the weights of the form $(\pm n, \pm m, \pm\ell)$,
for some fixed $n, m,\ell \in \N$ (and up to $8$ possible signs), not all equal, are typical.

By  Lemmas \ref{lem:central}, \ref{lem:oddhom} and \ref{lem:evenhom}  we conclude that the subcategories $\Bl_0$ and $\Bl_{n,m,\ell}$, for $n,m,\ell\in\N$ not all equal, are indecomposable and hence they are blocks. Furthermore, they are clearly pairwise distinct.
\end{proof}

\subsection{Typical blocks}\label{sec:block:typ}

The typical blocks in $\OO$ are very easy to describe completely.
Each typical block $\Bl_{n,m,\ell}$ as  in Proposition~\ref{prop:blocks:irr} has $2^r$ simple modules, where
\begin{equation}
  \label{eq:r}
  r =r(n,m,\ell) \in \{1,2,3\}
\end{equation}
is the number of nonzero integers (counted with multiplicities) among $\{n, m,\ell\}$. For example, we have $r(1,1,0) =2$.

\begin{prop} [Gorelik \cite{Gor02}]
  \label{prop:typical}
Let $n, m,\ell \in \N$, not all equal, and let $r$ be given as in \eqref{eq:r}.
Then the typical block $\Bl_{n,m,\ell}$ is equivalent to the principal block for the Lie algebra which is a direct sum of $r$ copies of $\sll$.
\end{prop}

Proposition~\ref{prop:typical} follows from the work of Gorelik \cite{Gor02} on strongly typical blocks of general type II Lie superalgebras, where for $D(2|1;\ka)$ the notions of typical and strongly typical coincide.
In our setting of $\D$, Proposition~\ref{prop:typical} can be proved directly by using translation functors
(this will be an easier version of the application of translation functors below, and so we shall not get into the detail here). Indeed the Verma flags of tilting modules and projective modules in $\Bl_{n,m,\ell}$ can be computed explicitly (which fits well with Proposition~\ref{prop:typical}), and this is all we need to know about typical blocks in this paper.

We recall that there are two tilting modules in  the principal block for $\sll$: the Verma module of anti-dominant highest weight, and the indecomposable module with a Verma flag of length two. Therefore the Verma flag structures of tilting modules in any typical block $\Bl_{n,m,\ell}$ follow from this and Proposition~\ref{prop:typical}. We will take this for granted and will not refer to this simple fact explicitly every time we use it.

\subsection{Tilting modules in the principal block $\Bl_0$}
 \label{subsec:tilting:irr}

The translation functors below from a typical block
to the atypical block $\Bl_0$ are obtained by tensoring with the adjoint module.

Let us describe the simple strategy which we use to obtain the Verma flags of tilting modules below.
The basic fact we use is that a translation functor $\E$ sends a tilting module $T$ to
a direct sum of tilting modules.
When a Verma flag of $\TT{}$ is known,  a Verma flag of $\E\TT{}$ can be read off easily.
The main point is to choose a suitable translation functor $\E$ applying to
a suitable titling module $\TT{}$ so the resulting module $\E\TT{}$ is indecomposable
(and hence it must be a  tilting module).
It turns out that we can always do that (with one exception),
largely thanks to the fact that $\D$ is of low rank $3$.
We shall apply the duality \eqref{tiltingD} repeatedly.

We introduce the following shorthand notations, for $n\ge 1$:
\[
\TT{n}^{\pm\pm\pm}  = \TT{\pm n, \pm n, \pm n}, \quad
\M{n}^{\pm\pm\pm}  = \M{\pm n, \pm n, \pm n}, \quad
\PP{n}^{\pm\pm\pm}  = \PP{\pm n, \pm n, \pm n}, \quad
\LL{n}^{\pm\pm\pm}  = \LL{\pm n, \pm n, \pm n}.
\]
We shall write $\TT{\la} =\sum_{\mu} t_{\la\mu} \M{\mu}$ to denote that
$\TT{\la}$ admits a Verma flag with $(\TT{\la}:\M{\mu})=t_{\la\mu}$.
The Verma flags of tilting modules are arranged in a weight decreasing order.

\begin{thm}
  \label{thm:tilting1}
For $n\ge 1$, we have the following Verma flags for tilting modules in $\Bl_0$:
\begin{align*}
\TT{n}^{---}  =
& \M{n}^{---}   + \M{n+1}^{---},
\\
\TT{n}^{--+}  =
& \M{n}^{--+} +\M{n}^{---}  +\M{n+1}^{--+}   + \M{n+1}^{---},
\\
\TT{n}^{-+-}  =
& \M{n}^{-+-} +\M{n}^{---}  +\M{n+1}^{-+-}   + \M{n+1}^{---},
\\
\TT{n}^{-++}  =
& \M{n}^{-++} +\M{n}^{-+-} + \M{n}^{--+}  +\M{n}^{---}
 +\M{n+1}^{-++}    +\M{n+1}^{-+-}   +\M{n+1}^{--+}   + \M{n+1}^{---}.
\end{align*}
\end{thm}

\begin{proof}
We prove the four formulae case-by-case. We choose to use the traditional indexing convention of modules in the proof.

(1)  Let $\mathcal E_1$ be the translation functor, obtained by first tensoring an object in $\mathcal O$ with the adjoint module and then projecting to $\mathcal B_0$. Applying $\E_1$ to $\TT{-n-2, -n, -n} =\M{-n-2, -n, -n}$, we obtain
$\E_1 \M{-n-2, -n, -n} =\M{-n, -n, -n} +\M{-n-1, -n-1, -n-1}$.
Note $(\TT{-n,-n,-n}: \M{-n-1, -n-1, -n-1}) =[\M{n+1,n+1,n+1}: \LL{n,n,n}] > 0$, by Lemma~\ref{lem:oddhom}.
Hence we must have $\TT{-n,-n,-n} =\E_1 \M{-n-2, -n, -n}$.

(2)
Applying the same translation functor $\E_1$ to $\TT{-n-2, -n, n}=\M{-n-2, -n, n} +\M{-n-2, -n, -n}$ (of typical highest weight), we obtain
\[
\E_1 \TT{-n-2, -n, n}=
\M{-n, -n, n} + \M{-n, -n, -n} +\M{-n-1, -n-1, n+1} + \M{-n-1, -n-1, -n-1}.
\]
Note $(\TT{-n,-n,n}: \M{-n-1, -n-1, -n-1}) =[\M{n+1,n+1,n+1}: \LL{n,n,-n}]> 0$,
since $f_2^n u^+_{n,n,n}$ is a singular vector in $\M{n+1,n+1,n+1}$,
where $u^+_{n,n,n}$ is the singular vector in $\M{n+1,n+1,n+1}$ of $\rho$-shifted weight $(n,n,n)$ in Lemma~\ref{lem:oddhom}.
Also by Lemma~\ref{lem:oddhom} and  \eqref{tiltingD} we obtain
$(\TT{-n,-n,n}: \M{-n-1, -n-1, n+1}) > 0$.
Finally, the remaining Verma $\M{-n, -n, -n}$ cannot be a direct summand of $\E_1 \TT{-n-2, -n, n}$ as it is not tilting by (1).
Hence we have $\TT{-n,-n,n} = \E_1 \M{-n-2, -n, n}$.

(3)
The formula for $\TT{-n, n, -n}$ is obtained by 
an argument parallel to the case (2) above.

(4)
Applying the translation functor $\E_1$ to $\TT{-n-2, n, n}$ (of typical highest weight), we obtain
\begin{align}
 \label{T-nnn}
\E_1 & \TT{-n-2, n, n}
=\M{-n, n, n} +\M{-n, n, -n} + \M{-n, -n, n}  +\M{-n, -n, -n}&\\
 & +\M{-n-1, n+1, n+1}   +\M{-n-1, n+1, -n-1} +\M{-n-1, -n-1, n+1}
 + \M{-n-1, -n-1, -n-1}.
 \notag
\end{align}
Note $(\TT{-n,n,n}: \M{-n-1, -n-1, -n-1}) =[\M{n+1,n+1,n+1}, \LL{n,-n,-n}]> 0$,
since $f_1^nf_2^n u^+_{n,n,n}$ is a singular vector in $\M{n+1,n+1,n+1}$,
where $u^+_{n,n,n}$ is the singular vector in $\M{n+1,n+1,n+1}$ of $\rho$-shifted weight $(n,n,n)$ in Lemma~\ref{lem:oddhom}.

Also $(\TT{-n,n,n}: \M{-n-1, n+1, n+1}) =[\M{n+1,-n-1,-n-1}, \LL{n,-n,-n}]> 0$, by Lemma~\ref{lem:oddhom}.

We claim that other Verma modules in \eqref{T-nnn}  including
$\M{-n, n, -n}, \M{-n, -n, n}, \M{-n, -n, -n},$ $\M{-n-1, n+1, -n-1}$ or $\M{-n-1, -n-1, n+1}$
 cannot appear as a leading term of a direct summand (i.e., a tilting submodule)
 of $\E_1 \TT{-n-2, n, n}$. Otherwise, the corresponding
 tilting character (known by (1)--(3) above) has its lowest term $\M{-n-1, -n-1, -n-1}$,
which  contradicts the fact that $(\TT{-n,n,n}: \M{-n-1, -n-1, -n-1}) > 0$,
 or else has its lowest term which does not appear in \eqref{T-nnn}, which is absurd.

Therefore $\E_1 \TT{-n-2, n, n}$ is indecomposable, and hence, we have $\TT{-n,n,n} = \E_1 \M{-n-2, n, n}$.
\end{proof}

\begin{thm}
  \label{thm:tilting2}
For $n\ge 2$, we have the following Verma flags of tilting modules in $\Bl_0$:
\begin{align}
\TT{n}^{+--}  =
& \M{n}^{+--} +\M{n-1}^{+--}  +\M{n-1}^{---}   + \M{n}^{---},
\notag \\ %
\TT{n}^{+-+}  =
& \M{n}^{+-+} +\M{n}^{+--} + \M{n-1}^{+-+}  +\M{n-1}^{+--}
 +\M{n-1}^{--+}    +\M{n-1}^{---}   +\M{n}^{--+}   + \M{n}^{---},
\notag \\ %
\TT{n}^{++-}  =
& \M{n}^{++-} +\M{n}^{+--} + \M{n-1}^{++-}  +\M{n-1}^{+--}
 +\M{n-1}^{-+-}    +\M{n-1}^{---}   +\M{n}^{-+-}   + \M{n}^{---},
\notag \\ %
\TT{n}^{+++}  =
& \M{n}^{+++} +\M{n}^{++-} +\M{n}^{+-+} +\M{n}^{+--}
+ \M{n-1}^{+++}  +\M{n-1}^{++-}   +\M{n-1}^{+-+}    +\M{n-1}^{+--}
\label{Tnnn}
\\
+& \M{n-1}^{-++}  +\M{n-1}^{-+-}   +\M{n-1}^{--+}    +\M{n-1}^{---}
+\M{n}^{-++}   + \M{n}^{-+-}+\M{n}^{--+}   + \M{n}^{---}.
\notag
\end{align}
\end{thm}

\begin{proof}
We use the traditional indexing convention of modules in the proof.

The formula for $\TT{n, -n, -n}$ is obtained by applying a suitable translation functor to $\TT{n-2, -n, -n}$, where
the case of $n=2$ needs to be separately treated.
The formula for $\TT{n, -n, n}$ is obtained by applying a suitable translation functor to $\TT{n-2, -n, n}$.
The formula for $\TT{n, n, -n}$ is obtained by applying a suitable translation functor to $\TT{n-2, n, -n}$.
The formula for $\TT{n, n, n}$ is obtained by applying a suitable translation functor to $\TT{n-2, n, n}$.

It remains to show that the resulting 4 modules are indecomposable. To do so, we proceed case-by-case in the natural order,
and the argument for a given case uses the preceding formulae.

We only provide a full detail in the last case of $\TT{n, n, n}$, as all cases are similar and
the strategy is basically the same as for Case (4) in the proof of Theorem~\ref{thm:tilting1}.
We shall assume that the previous 3 formulae have been established.
One checks that $\E_4\TT{n-2, n, n}$ (the resulting module in the atypical block by applying a suitable translation functor $\E_4$)
has a Verma flag of length 16 as on the right hand side of \eqref{Tnnn}.
We shall argue that none of the 16 Verma modules except
$\M{n, n, n}$ can be the highest term of a direct summand (i.e., a tilting submodule) of $\E_4\TT{n-2, n, n}$.

First, we have
$(\TT{n,n,n}: \M{-n, -n, -n}) =[\M{n,n,n}: \LL{-n,-n,-n}] > 0$, since $f_1^n f_2^n u^+_{-n,n,n}$ is a singular vector in $\M{n,n,n}$,
where $u_{-n,n,n}^+$ is a singular vector in $\M{n,n,n}$ of $\rho$-shifted weight $(-n,n,n)$ by Lemma~\ref{lem:evenhom}.

Also $(\TT{n,n,n}: \M{n-1, n-1, n-1}) =[\M{1-n,1-n,1-n}: \LL{-n,-n,-n}] > 0$, by Lemma~\ref{lem:oddhom}.

Now we observe that, other than $\M{n, n, n}, \M{n-1, n-1, n-1}$ and $\M{-n, -n, -n}$, the other 13 Verma modules
in \eqref{Tnnn} cannot appear as a leading term of a tilting submodule of $\E_4\TT{n-2, n, n}$,
since the corresponding tilting character (known by the preceding 3 formulae) has its lowest term
$\M{-n, -n, -n}$ or a Verma which does not appear in \eqref{Tnnn}.

Hence $\E_4\TT{n-2, n, n}$ is indecomposable, and we have $\TT{n,n,n}=\E_4\TT{n-2, n, n}$.
\end{proof}

\subsection{Tilting modules with irregular Verma flags}

It remains to determine the Verma flag multiplicites in the 5 remaining tilting modules:
$\TT{0,0,0}$ and $\TT{1, \pm 1, \pm 1}$.

\begin{thm} 
 \label{tilting:extra}
We have the following Verma flags for the tilting modules in $\Bl_0$:
\begin{align*}
\TT{000}  =
& \M{000} +\M{1}^{-++}  +\M{1}^{-+-}   + \M{1}^{--+} + \M{1}^{---},
 \\ %
\TT{1}^{+--}  =
&  \M{1}^{+--}  + \M{000}   + \M{1}^{-+-} + \M{1}^{--+}  + {2\M{1}^{---} + \M{2}^{---}, }
  \\ %
\TT{1}^{+-+}  =
& \M{1}^{+-+} +\M{1}^{+--}  +\M{000}   + \M{1}^{--+} + \M{1}^{---},
 \\ %
\TT{1}^{++-}  =
& \M{1}^{++-} +\M{1}^{+--}  +\M{000}   + \M{1}^{-+-} + \M{1}^{---},
 \\ %
\TT{1}^{+++}  =
& \M{1}^{+++} +\M{1}^{++-}  + \M{1}^{+-+} +\M{1}^{+--}  + 2\M{000}   + \M{1}^{-++} + \M{1}^{-+-}  + \M{1}^{--+} + \M{1}^{---}.
\end{align*}
\end{thm}

\begin{proof}
We use the traditional indexing convention of modules in the proof.

The translation functors below from a typical block to the atypical block $\Bl_0$
are obtained from tensoring with the adjoint representation, except in the case for $\TT{1,-1,-1}$ (i.e., $\TT{1}^{+--}$). We shall treat this case at last.

Applying a translation functor $\E_0$ to $\TT{-2,0,0} =\M{-2,0,0}$, we obtain a module with the following Verma flag
\begin{equation}
 \label{eq:T000}
 \E_0 \M{-2,0,0} =
\M{0,0,0} + \M{-1,1,1} + \M{-1,1,-1} +  \M{-1,-1,1} +\M{-1,-1,-1}.
\end{equation}
Since we know the tilting module indexed by  either
$(-1,1,1), (-1,1,-1), (-1,-1,1)$ or $(-1,-1,-1)$ contains the Verma $\M{-2,-2,-2}$ (which is not present in \eqref{eq:T000}),
we conclude that \eqref{eq:T000} must be a single tilting module with highest weight $(0,0,0)$, i.e., $\TT{0,0,0}$.

Applying a translation functor $\E_1$ to $\TT{0,-2,0} =\M{0,-2,0}$, we obtain
\begin{equation}
 \label{eq:T1-11}
\E_1 \M{0,-2,0} =
\M{1,-1,1} +  \M{1,-1,-1} +\M{0,0,0} + \M{-1,-1,1} + \M{-1,-1,-1}.
\end{equation}
Note $(\TT{1,-1,1}: \M{1,-1,-1}) =[\M{-1,1,1}: \LL{-1,1,-1}] > 0$,
since there is a singular vector $f_2^2 v^+_{-1,1,1}$ in $\M{-1,1,1}$.
The tilting module indexed by either
$(-1,-1,1)$, or $(-1,-1,-1)$ contains the term $\M{-2,-2,-2}$ and $\TT{0,0,0}$ contains the term $\M{-1,1,1}$ (not present in \eqref{eq:T1-11}).
Therefore we conclude that \eqref{eq:T1-11} must be a single tilting module,  which is $\TT{1,-1,1}$.

A completely analogous argument by applying a suitable translation functor $\E_2$ to $\TT{0,0,-2} =\M{0,0,-2}$ gives us
the tilting module $\TT{1,1,-1}$:
\begin{equation}
 \label{eq:T11-}
\TT{1,1,-1} =
\E_2 \M{0,0,-2}
= \M{1,1,-1} +   \M{1,-1,-1} + \M{0,0,0} + \M{-1,1,-1} + \M{-1,-1,-1}.
\end{equation}

Now we shall apply a suitable translation functor $\E_3$ to the following short exact sequence
of $\g$-modules
\[
0\longrightarrow \M{2,0,0} \longrightarrow \TT{2,0,0} \longrightarrow \M{-2,0,0} \longrightarrow 0.
\]
 Recalling
that $\TT{0,0,0}$ from \eqref{eq:T000} is obtained by applying $\E_3$ (=$\E_0$) to $\M{-2,0,0}$, we
have obtained a short exact sequence
\begin{equation}
  \label{eq:ses}
  0\longrightarrow \E_3 \M{2,0,0} \longrightarrow \E_3 \TT{2,0,0} \longrightarrow \TT{0,0,0} \longrightarrow 0.
\end{equation}
We obtain
\begin{align}
 \label{eq:T111}
   \begin{split}
\E_3 \TT{2,0,0} &=
\M{1,1,1} +  \M{1,1,-1}   + \M{1,-1,1} +  \M{1,-1,-1} + \M{0,0,0} +   \TT{0,0,0}
\\
 &=
\M{1,1,1} +  \M{1,1,-1}   + \M{1,-1,1} +  \M{1,-1,-1}
\\ & \qquad
+ 2\M{0,0,0}  + \M{-1,1,1} + \M{-1,1,-1} +  \M{-1,-1,1} +\M{-1,-1,-1}.
  \end{split}
\end{align}
Note $\E_3 \TT{2,0,0}$ must be a direct sum of tilting modules,
and we shall show that $\E_3 \TT{2,0,0}$  is indecomposable by arguing no Verma in \eqref{eq:T111} except $\M{1,1,1}$
can appear as a highest term in a direct summand.
Indeed, the tilting module indexed by either $(1,-1,1)$, or $(1,-1,-1)$ contains $\M{-2,-2,-2}$.
Note $(\TT{1,1,1}: \M{1,1,-1}) =[\M{-1,-1,1}: \LL{-1,-1,-1}] > 0$, by Lemma~\ref{lem:oddhom}.
Also, $(\TT{1,1,1}: \M{0,0,0}) =[\M{0,0,0}: \LL{-1,-1,-1}] > 0$, by Lemma~\ref{lem:oddhom}.
Similarly, $(\TT{1,1,1}: \M{-1,1,1}) =[\M{1,-1,-1}: \LL{-1,-1,-1}] > 0$,  by Lemma~\ref{lem:evenhom}.
By \eqref{eq:ses} $\TT{0,0,0}$ appears as a quotient module of $\E_3 \TT{2,0,0}$ and it is indecomposable.
This shows none of the Verma modules in \eqref{eq:T111} except $\M{1,1,1}$
appears as highest weight term in a direct summand of \eqref{eq:T111};
that is, $\E_3 \TT{2,0,0}$ is indecomposable and must be $\TT{1,1,1}$.

\vspace{2mm}
It takes some extra effort to establish the Verma flag structure for the remaining tilting module $\TT{1,-1,-1}$.
It turns out to be less effective to apply the translation functor from tensoring
with the adjoint module to $\TT{-1,-1,-1} =\M{-1,-1,-1} +\M{-2,-2,-2}$ (of atypical highest weight).
Instead, we shall apply the translation functor $\E_4$ from tensoring
with the module $\LL{1,2,1}$ of highest weight $(2,1,0)$  
to $\TT{-1,-2,-1}$, which produces a module with a Verma flag of smaller length.
The module $\LL{1,2,1}$ has dimension $32$, and its weights are (here it is understood that we mix all possible positive/negative sign combinations):
\begin{align*}
&(\pm 2,\pm 1, 0),\quad  (\pm 1, \pm 2, \pm 1), \quad (0, \pm 3, 0),  \quad  (0, \pm 1, \pm 2),
\\
&\quad (\pm 1, 0, \pm 1) \; [\text{of multiplicity }2],  \quad
(0, \pm 1, 0) \; [\text{of multiplicity }3].
\end{align*}
Applying a suitable translation functor $\E_4$ to $\TT{-1,-2,-1} =\M{-1,-2,-1}$, we obtain
\begin{align}
 \label{eq:T1--}
\E_4 \TT{-1,-2,-1} =
& \M{1,-1,-1} +  \M{0,0,0}   +  \M{-1,1,-1} + \M{-1,-1,1}
+ 3 \M{-1,-1,-1} + 2\M{-2,-2,-2}.
\end{align}

It follows from Lemma \ref{lem:ML3} below that $(\TT{1,-1,-1} : \M{-1,-1,-1})=2$ and hence $\E_4 \TT{-1,-2,-1}=T_{1,-1,-1}\oplus T_{-1,-1,-1}$.
Since $\TT{-1,-1,-1} =\M{-1,-1,-1} + \M{-2,-2,-2}$ by Theorem~\ref{thm:tilting1}, the desired Verma flag formula for $\TT{1,-1,-1}$ follows from \eqref{eq:T1--}.
\end{proof}


\begin{lem}
  \label{lem:ML3}
In $\Bl_0$, we have $[\M{1,1,1}:\LL{-1,1,1}]=2$.
\end{lem}

\begin{proof}
Let us first introduce some notation. We shall write $\M{\mu} =\sum_{\la} p_{\la\mu} \LL{\la}$  to denote that $\LL{\la}$ is a composition factor of $\M{\mu}$ with multiplicity $p_{\la\mu}$. For $X=M,L$ denote by $X_\la^\mu$ the $\left(\mu-\rho\right)$-weight space of $X_\la$. Recall $\rho=(-1,1,1),$ and hence $X_\la^{\rho}$ denote the zero-weight subspace.

{\bf Claim.}  The modules $\LL{1,-1,1},\LL{1,1,-1},\LL{1,-1,-1}$ and $\LL{0,0,0}$ all have trivial zero-weight spaces, that is,
$\LL{1,-1,1}^{\rho}= \LL{1,1,-1}^{\rho}= \LL{1,-1,-1}^{\rho}=\LL{0,0,0}^{\rho}=0.$

Let us prove the claim. First, using the Verma flag structures of the tilting modules in $\Bl_0$ which we have established so far, together with the duality
\eqref{tilt=proj},
we obtain the following:
\begin{align}\label{eq:aux1}
\begin{split}
\M{1,-1,1} =& \LL{1,-1,1} +\LL{1,-1,-1} + \LL{0,0,0}  + \LL{-1,1,1} + \LL{-1,-1,1} + \LL{-1,-1,-1}  + \LL{-2,-2,2} + \LL{-2,-2,-2},
 \\
\M{1,1,-1} =& \LL{1,1,-1} +\LL{1,-1,-1} + \LL{0,0,0}  + \LL{-1,1,1} + \LL{-1,1,-1} + \LL{-1,-1,-1}  +\LL{-2,2,-2} +\LL{-2,-2,-2}.
\end{split}
\end{align}
Now observe that $\dim \M{1,-1,1}^{\rho}
=1$. Since $[\M{1,-1,1}:\LL{-1,1,1}]=1$ by \eqref{eq:aux1}, this implies that $\LL{1,-1,1}^{\rho}=0$. An identical argument  shows that $\LL{1,1,-1}^{\rho}=0$.

On the other hand, we have $\M{1,-1,-1}^{\rho}=0$ and so $\LL{1,-1,-1}^{\rho}=0$.
Finally, we have $\LL{0,0,0}^{\rho}=0$, since $f_0v_{(0,0,0)}$ is singular in $\M{0,0,0}$ and $\dim \M{0,0,0}^{\rho}=1$. This proves the claim.

Next, again using the Verma flag structures of the tilting modules that we have proved so far together with  the duality
\eqref{tilt=proj} we can list the possible composition factors in the Verma module $\M{1,1,1}$ (ignoring multiplicities) that have highest weights greater than or equal to the zero weight:
\begin{align*}
\LL{1,1,1},\quad \LL{1,-1,1},\quad \LL{1,1,-1},\quad \LL{1,-1,-1},\quad \LL{0,0,0},\quad \LL{-1,1,1}.
\end{align*}
We shall compare the dimension of the zero-weight subspaces.
A direct computation gives us $\dim \M{1,1,1}^{\rho}=5$, while $\dim \LL{1,1,1}^{\rho}=3$ (which follows from the fact that $\LL{1,1,1}$ is the adjoint representation).

Hence, combining with the claim above, we conclude that the multiplicity of
the composition factor $\LL{-1,1,1}$ (which is the trivial module) in $\M{1,1,1}$ is $5-3=2$.
\end{proof}

\begin{rem}
From Theorem \ref{tilting:extra} we have $(\TT{1,-1,1}:\M{-2,-2,-2})=0$, which by \eqref{tiltingD} gives
\begin{align}\label{aux001}
[\M{2,2,2}:\LL{-1,1,-1}]=0.
\end{align}
By Lemmas \ref{lem:oddhom} and \ref{lem:evenhom} we have
\begin{align}\label{nonzerohom}
\text{Hom}_\g(\M{-1,1,-1},\M{-1,1,1})\not=0,\quad \text{Hom}_\g(\M{-1,1,1},\M{1,1,1})\not=0,\quad \text{Hom}_\g(\M{1,1,1},\M{2,2,2})\not=0.
\end{align}
Recall that $v^+_\la$ denotes a nonzero highest weight vector of highest weight $\la-\rho$ in $\M{\la}$. Then we have $f_2v^+_{-1,1,1}$ is a nonzero singular vector in $\M{-1,1,1}$ of weight $(-1,1,-1)-\rho$. We denote by $u v^+_{1,1,1}$ a nonzero singular vector in $\M{1,1,1}$ of highest weight $(-1,1,1)-\rho$, and by $u' v^+_{2,2,2}$ a nonzero singular vector in $\M{2,2,2}$ of highest weight $(1,1,1)-\rho$, where $u,u'\in U(\mathfrak{n}_-)$. These vectors exist by \eqref{nonzerohom}. Composing these nonzero homomorphisms between Verma modules we get a homomorphism in $\text{Hom}(\M{-1,1,-1},\M{2,2,2})$ given by sending
\begin{align*}
v^+_{-1,1,-1}\in\M{-1,1,-1}\rightarrow f_2v^+_{-1,1,1}\in\M{-1,1,1}\rightarrow f_2u u'v_{2,2,2}\in\M{2,2,2}.
\end{align*}
By \eqref{aux001} $f_2u u'=0$, and hence $uu'=0$, as $f_2$ is not a zero divisor. Therefore, $u\in U(\mathfrak{n}_-)$ (which gives us the Verma module homomorphism associated to the even reflection $s_{2\delta}$) is a (left) zero divisor. This is the  first such an example in super representation theory which we are aware of.
\end{rem}

\subsection{Characters of projectives in $\Bl_0$}

Recall that projective modules in $\OO$ admit Verma flags.
The tilting character formulae in Theorems~\ref{thm:tilting1}--\ref{tilting:extra}
can be readily reformulated using \eqref{tilt=proj} into character formulae for projective modules in Propositions~\ref{proj1:irr}--\ref{proj2:irr} below. We shall write $\PP{\la} =\sum_{\mu} p_{\la\mu} \M{\mu}$ to indicate that
$\PP{\la}$ admits a Verma flag with $(\PP{\la}:\M{\mu})=p_{\la\mu}$.
The Verma flags of projective modules are arranged in a weight increasing order.

\begin{prop}
 \label{proj1:irr}
For $n\ge 1$, we have the following Verma flags for projective modules:
\begin{align*}
\PP{n}^{+++} =& \M{n}^{+++} +\M{n+1}^{+++},
 \\
\PP{n}^{++-} =& \M{n}^{++-} + \M{n}^{+++} +\M{n+1}^{++-} + \M{n+1}^{+++},
 \\
\PP{n}^{+-+} =& \M{n}^{+-+} + \M{n}^{+++} +\M{n+1}^{+-+} + \M{n+1}^{+++},
 \\
\PP{n}^{+--}  =& \M{n}^{+--} + \M{n}^{+-+} + \M{n}^{++-} +  \M{n}^{+++}
  +\M{n+1}^{+--} +\M{n+1}^{+-+} +\M{n+1}^{++-} + \M{n+1}^{+++}.
\end{align*}
\end{prop}

\begin{prop}
 \label{proj2:irr}
For $n\ge 2$, we have the following Verma flags for projective modules:
\begin{align*}
\PP{n}^{-++} =& \M{n}^{-++} +\M{n-1}^{-++} + \M{n-1}^{+++} + \M{n}^{+++},
 \\
\PP{n}^{-+-}  =& \M{n}^{-+-} +\M{n}^{-++} + \M{n-1}^{-+-}  + \M{n-1}^{-++}
 + \M{n-1}^{++-}  + \M{n-1}^{+++} +\M{n}^{++-} + \M{n}^{+++},
\\
\PP{n}^{--+}  =& \M{n}^{--+} + \M{n}^{-++} + \M{n-1}^{--+} + \M{n-1}^{-++}
 + \M{n-1}^{+-+}  + \M{n-1}^{+++} +\M{n}^{+-+} + \M{n}^{+++},
 \\
\PP{n}^{---}  =& \M{n}^{---} + \M{n}^{--+} + \M{n}^{-+-} + \M{n}^{-++}
+\M{n-1}^{---}  +\M{n-1}^{--+} + \M{n-1}^{-+-}  + \M{n-1}^{-++}   &\\
&
+\M{n-1}^{+--} +\M{n-1}^{+-+} +\M{n-1}^{++-}  + \M{n-1}^{+++}
+\M{n}^{+--}  +\M{n}^{+-+} +\M{n}^{++-}  +\M{n}^{+++}.
\end{align*}
\end{prop}

The following projective modules admit irregular Verma flags.
\begin{prop} 
 \label{proj3:irr}
We have the following Verma flags for the projective modules:
\begin{align*}
\PP{000} =& \M{000} +\M{1}^{+--}  + \M{1}^{+-+} + \M{1}^{++-} +\M{1}^{+++},
\\
\PP{1}^{-++} =& \M{1}^{-++} + \M{000} + \M{1}^{+-+}   + \M{1}^{++-} + {2\M{1}^{+++} +  \M{2}^{+++}},
 \\
\PP{1}^{-+-} =& \M{1}^{-+-} +\M{1}^{-++} +\M{000} +\M{1}^{++-}  +\M{1}^{+++},
 \\
\PP{1}^{--+} =& \M{1}^{--+} +\M{1}^{-++} +\M{000} +\M{1}^{+-+}  +\M{1}^{+++},
 \\
\PP{1}^{---} =& \M{1}^{---} + \M{1}^{--+} +\M{1}^{-+-} +\M{1}^{-++}
+2 \M{000} + \M{1}^{+--}  + \M{1}^{+-+} + \M{1}^{++-} +\M{1}^{+++}.
\end{align*}
\end{prop}

\subsection{Projective tilting modules in $\Bl_0$}

We are interested in knowing which tilting modules are projective.
Let us first examine the lists of tilting modules and projective modules with regular flags.
Note that these tilting modules always have lowest terms $\M{n}^{---}$ for $n\ge 2$, while the projective modules always have highest terms $\M{n}^{+++}$ for $n\ge 2$.
Hence the only possible isomorphism is between
$\TT{n}^{+++} $ and $\PP{n}^{---}$ for $n\ge 2$, whose Verma flags of length $16$ indeed match perfectly.

Among the 5 tilting modules with irregular flags, $\TT{1}^{+--}$ cannot be projective as it cannot match with the Verma flags of possible candidates $\PP{2}^{---}$ or $\PP{1}^{---}$.
Among the 5 projective modules with irregular flags, $\PP{1}^{-++}$ cannot be tilting as it cannot match with the Verma flags of possible candidates $\TT{2}^{+++}$ or $\TT{1}^{+++}$.
The other 4 tilting modules all have lowest term $\M{1}^{---}$, while the other 4 projective modules all have highest term $\M{1}^{+++}$.
So the only possible isomorphism is between
$\TT{1}^{+++} $ and $\PP{1}^{---}$, whose Verma flags of length $9$ indeed match perfectly.

\begin{thm}
 \label{T=P:irrational}
We have an isomorphism of modules $\TT{n}^{+++} \cong \PP{n}^{---}$, for $n\ge 1$.
Furthermore, there are no other projective tilting modules of atypical weights in $\Bl_0$.
\end{thm}

\begin{proof}
Thanks to our discussion above it suffices to show that the tilting modules $\TT{n}^{+++}$ are indeed projective, for $n\ge 1$. To that end, we observe that the initial tilting modules to which we applied translation functors to obtain $\TT{n}^{+++}$ (for $n\ge 2$ in Theorem~ \ref{thm:tilting2} and for $n=1$ in Theorem~ \ref{tilting:extra}) are typical and projective by Proposition~\ref{prop:typical}.
Since translation functors are exact, the resulting tilting modules $\TT{n}^{+++}$ are also projective.
\end{proof}

Note that Theorem \ref{T=P:irrational} implies that the tilting module $\TT{n}^{+++}$ has a simple head isomorphic to $\LL{n}^{---}$, for $n\ge 1$.

\subsection{Composition factors of Verma modules in $\Bl_0$}

We shall simply write $\M{\mu} =\sum_{\la} p_{\la\mu} \LL{\la}$ below to denote the composition multiplicity $[\M{\mu}: \LL{\la}] =p_{\la\mu}$.
Using the BGG reciprocity
$(\PP{\la} : \LL{\mu}) = [\M{\mu}: \LL{\la}]$, we obtain the formulae for composition factors in Verma modules.
The composition factors of Verma modules are arranged in a weight decreasing order.

\begin{prop}
For $n\ge 1$, we have the following composition factors of Verma modules:
\begin{align*}
\M{n}^{---} =& \LL{n}^{---} +\LL{n+1}^{---},
 \\
\M{n}^{--+} =& \LL{n}^{--+} + \LL{n}^{---} +\LL{n+1}^{--+} + \LL{n+1}^{---},
 \\
\M{n}^{-+-} =& \LL{n}^{-+-} + \LL{n}^{---} +\LL{n+1}^{-+-} + \LL{n+1}^{---},
 \\
\M{n}^{-++}  =& \LL{n}^{-++} +\LL{n}^{-+-} + \LL{n}^{--+}  +\LL{n}^{---}
+\LL{n+1}^{-++}  +\LL{n+1}^{-+-} +\LL{n+1}^{--+} + \LL{n+1}^{---}.
\end{align*}
\end{prop}

\begin{prop}
For $n\ge 2$, we have the following composition factors of Verma modules:
\begin{align*}
\M{n}^{+--} =& \LL{n}^{+--} +\LL{n-1}^{+--}  + \LL{n}^{---} + \LL{n+1}^{---},
 \\
\M{n}^{+-+}  =& \LL{n}^{+-+}  +\LL{n}^{+--}  + \LL{n-1}^{+-+}    + \LL{n-1}^{+--}  +\LL{n}^{--+}    + \LL{n}^{---} + \LL{n+1}^{--+} + \LL{n+1}^{---},
 \\
\M{n}^{++-}  =& \LL{n}^{++-}   +\LL{n}^{+--}  + \LL{n-1}^{++-}   + \LL{n-1}^{+--}   +\LL{n}^{-+-}   + \LL{n}^{---} + \LL{n+1}^{-+-}   + \LL{n+1}^{---},
 \\
\M{n}^{+++}  =& \LL{n}^{+++}  +\LL{n}^{++-} + \LL{n}^{+-+}  +\LL{n}^{+--}
   +\LL{n-1}^{+++}   +\LL{n-1}^{++-} + \LL{n-1}^{+-+}  +\LL{n-1}^{+--} &\\
+& \LL{n}^{-++}  +\LL{n}^{-+-}   +\LL{n}^{--+}    + \LL{n}^{---}
 +\LL{n+1}^{-++}  + \LL{n+1}^{-+-}  +\LL{n+1}^{--+}  +\LL{n+1}^{---} + {\delta_{n,2} \LL{1}^{-++}}.
\end{align*}
\end{prop}

The Verma modules with irregular composition series are given as follows.
\begin{prop}
We have the following composition factors of Verma modules:
\begin{align*}
\M{000} =& \LL{000} +\LL{1}^{-++}  + \LL{1}^{-+-} + \LL{1}^{--+} + 2\LL{1}^{---},
 \\
\M{1}^{+--} =& \LL{1}^{+--} + \LL{000} + \LL{1}^{---} +  \LL{2}^{---},
 \\
\M{1}^{+-+} =& \LL{1}^{+-+} +\LL{1}^{+--} + \LL{000}  + \LL{1}^{-++} + \LL{1}^{--+} + \LL{1}^{---}  +\LL{2}^{--+} +\LL{2}^{---},
 \\
\M{1}^{++-} =& \LL{1}^{++-} +\LL{1}^{+--} + \LL{000}  + \LL{1}^{-++} + \LL{1}^{-+-} + \LL{1}^{---} +\LL{2}^{-+-} + \LL{2}^{---},
 \\
\M{1}^{+++} =& \LL{1}^{+++} +\LL{1}^{++-} +\LL{1}^{+-+} +\LL{1}^{+--} + \LL{000}  + {2 \LL{1}^{-++}} + \LL{1}^{-+-} + \LL{1}^{--+} + \LL{1}^{---} \\
+ & \LL{2}^{-++} + \LL{2}^{-+-} + \LL{2}^{--+} + \LL{2}^{---}.
\end{align*}
\end{prop}

 \section{Character formulae in $\OO$, for $\ka \in \Q$}
  \label{sec:rational}

 In this section, we assume  $\ka   \in \Q \backslash \{0,-1\}$.
We can and will assume $\ka>0$ thanks to the isomorphisms \eqref{D:iso}. We shall obtain the Verma flag multiplicities of every tilting module and every projective module in $\OO$.
 We also describe the composition factors of every Verma module.

\subsection{Classification of blocks}

We denote
 \[
 \ka = p/d,
 \qquad \text{for } p,d\in \Z_{>0}, \;(p,d)=1.
 \]
Recall a $\rho$-shifted weight $(x,y,z)$ is atypical if it satisfies \eqref{atypical}.
A $\rho$-shifted weight is called {\em singular} if at least one of the 3 coordinates is zero, and a weight is
{\em singular} if its $\rho$-shifted weight is singular; otherwise it is called {\em regular}.

For $k \in \N$, we introduce
 the following $\rho$-shifted weights (with 8 choices of signs):
 \begin{align}
 f_{k;n}^{\pm\pm\pm} =
  \begin{cases}
 \big( {\pm(-n)},\; \pm(-n-kp),\; \pm(-n+kd) \big), & \text{ if } n \le -kp,
 \\
\big( {\pm(-n)},\; \pm(n+kp),\; \pm(-n+kd) \big), & \text{ if } 1-kp \le n \le 0,
 \\
\big( {\pm n},\; \pm(n+kp),\; \pm(-n+kd) \big), & \text{ if }  0 \le n \le kd-1,
 \\
\big( {\pm n},\; \pm(n+kp),\; \pm(n-kd) \big), & \text{ if } kd \le n.
 \end{cases}
\end{align}
For $k\ge 1$ and for each of 4 choices of signs, we have the following identification
\begin{align*}
f_{k;0}^{+\pm\pm} = f_{k;0}^{-\pm\pm} (\stackrel{\text{def}}{=:} f_{k;0}^{\circ\pm\pm}),
\quad
f_{k;-kp}^{\pm+\pm} = f_{k;-kp}^{\pm-\pm} (\stackrel{\text{def}}{=:} f_{k;0}^{\pm\circ\pm}),
\quad
f_{k;kd}^{\pm\pm+} = f_{k;kd}^{\pm\pm-} (\stackrel{\text{def}}{=:} f_{k;0}^{\pm\pm\circ}),
\end{align*}
and these identifications are the only possible ones among all weights $f_{k;n}^{\pm\pm\pm}$, for $k\ge 1$.

Denote $\Wt_0 =\{f_{0;n}^{\pm\pm\pm}\mid n\in \Z\}
=\{ (\pm n, \pm n, \pm n) \mid n\in \Z_{>0}\} \cup \{0,0,0\}$,
and
\[
\Wt_k =\big\{f_{k;n}^{\pm\pm\pm} \mid n\in \Z \backslash \{0, -kp, kd\} \big\}
 \sqcup \{f_{k;0}^{\circ\pm\pm}, \, f_{k;-kp}^{\pm\circ\pm}, \, f_{k;kd}^{\pm\pm\circ}\}, \quad \text{ for } k\ge 1.
\]
 We shall often omit the index $k$ when there is no confusion by setting
\[
 f_{n}^{\pm\pm\pm} = f_{k;n}^{\pm\pm\pm},
 \quad
 f_{0}^{\circ\pm\pm}=f_{k;0}^{\circ\pm\pm},
 \quad
 f_{-kp}^{\pm\circ\pm}=f_{k;-kp}^{\pm\circ\pm},
\quad
f_{kd}^{\pm\pm\circ}=f_{k;kd}^{\pm\pm\circ}.
\]

We denote by $X_{\text{aty}}$ the set of atypical weights in the weight lattice $X$.
 \begin{lem}
   \label{lem:atypical}
The set $X_{\text{aty}}$
 is identified with the disjoint union $\bigsqcup_{k\in \N} \Wt_k$ via the isomorphism  $X\cong \Z^3$ in \eqref{eq:lattice}.
 \end{lem}

\begin{proof}
Recall that $\ka=\frac{p}{d}\in\Q$ with $p,d>0$ and $(p,d)=1$.  Let $f=(n,m,\ell) \in\Z^3$ be an atypical $\rho$-shifted  weight. It follows by \eqref{atypical} that for some $\sigma,\tau\in\{+1,-1\}$ we have
\[
(-n+\sigma m)d=(n-\tau \ell)p.
\]
Since $(p,d)=1$ there exists $k\in\Z$ such that $(-n+\sigma m)=kp$, and hence $(n-\tau \ell)=kd$. Thus, we get
\begin{align*}
m=\pm(n+kp), \quad \ell=\pm(n-kd).
\end{align*}
It follows that any atypical weight is of the form
\begin{align*}
&f^{\pm\pm\pm}_{k;n}=(\pm|n|,\pm|n+ kp|,\pm|n- kd|),\\
&f^{\pm\pm\pm}_{k;-n}=(\pm|n|,\pm|n- kp|,\pm|n+ kd|),
\end{align*}
for $k,n\in\N$.
This completes the proof.
\end{proof}

The Bruhat order $\preceq$ on the lattice $\Z^3$ is defined as the transitive closure of the following relations:
\begin{enumerate}
\item
$(-a,b,c) \prec (a,b,c)$ if $a>0$,
\item
$(a,-b,c) \prec (a,b,c)$ if $b>0$,
\item
$(a,b,-c) \prec (a,b,c)$ if $c>0$,
\item
$f_{k;\pm (n+1)}^{- \sigma\tau}\prec f_{k;\pm n}^{- \sigma\tau}$,
and $f_{k;\pm n}^{+ \sigma\tau}\prec f_{k;\pm(n+1)}^{+ \sigma\tau}$, for any $k\ge 0$, $n\in\N$ and $\sigma, \tau=\pm$.
\end{enumerate}
This restricts to a Bruhat ordering on $\Wt_k$, for each $k \ge 0$.
Via the isomorphism $X \cong \Z^3$ in \eqref{eq:lattice}, this defines a Bruhat ordering on the weight lattice $X$. We note that the Bruhat ordering thus defined is consistent with the Harish-Chandra homomorphism; cf. \eqref{im:HC}. Furthermore, by Lemmas \ref{lem:oddhom} and \ref{lem:evenhom}, we have $f\prec g$ if and only if there exists a sequence of weights $f=f_1\prec f_2\prec \cdots\prec f_k=g$ such that $\text{Hom}_\g(M_{f_i},M_{f_{i+1}})\not=0$, for all $i=1,\cdots k-1$.

The typical blocks are controlled by the Weyl group $W$.
A typical block is the full subcategory of $\OO$ whose composition factors are of the form $\LL{f}$, where $f$ runs over a fixed $W$-orbit in $\Z^3 \backslash \bigsqcup_{k\in \N} \Wt_k$. There is a unique representative in $Wf$ of the form $(n,m,\ell)$, for $n,m,\ell \in \N$ and we denote this block by $\Bl_{n,m,\ell}$. A version of Proposition~\ref{prop:typical} holds (cf.~\cite{Gor02}), that is,
the typical block $\Bl_{n,m,\ell}$
is equivalent to the principal block for the Lie algebra which is a direct sum of $r$ copies of $\sll$, where $r$ is as in \eqref{eq:r}. We shall freely use the character formulae $\TT{f}$ for typical $f$ below.

From now on we shall focus on the atypical blocks in $\OO$.
\begin{prop}
  \label{prop:blocks}
For each $k\in \N$,
there is a block $\Bl_k$ in $\OO$ which contains simple modules $\LL{f}$, for $f \in \Wt_k$.
Moreover, any atypical block in $\OO$ is one of $\Bl_k$, for $k\in \N$.
\end{prop}

\begin{proof}
Let $\Bl_k$ be the full subcategory of $\OO$ of modules that have composition factors of the form $L_f$, for $f\in\Wt_k$.

Recall that the eigenvalue $c_f$ of the Casimir operator on $L_f$ is given by \eqref{action:cas}. Now for $\ka=\frac{p}{d}$ we have $c_{f^{\pm\pm\pm}_{k;n}}=k^2(p^2+pd)$, for $f^{\pm\pm\pm}_{k;n}\in\Wt_k$, and hence is independent of $n$. Clearly, for $k,\ell\in\N$, with $k\not=\ell$ we have $k^2(p^2+pd)\not=\ell^2(p^2+pd)$. Thus, we have established that the subcategories $\mathcal B_k$ are separated by the Casimir element.

By Lemmas~\ref{lem:oddhom} and \ref{lem:evenhom}, all the weights in $\Wt_k=\{f_{k,n}^{\pm\pm\pm}|n\in\Z\}$ are linked, for each fixed $k$. Hence $\Bl_k$ are indeed blocks. As we have exhausted the set of atypical weights by Lemma~\ref{lem:atypical}, the proposition is proved.
\end{proof}

\begin{rem}
The block $\Bl_0$ is the principal block which contains the trivial module. As seen above, the blocks $\Bl_k$, for $k\in\N$, are separated by the eigenvalues of the Casimir element, which acts on any simple module in $\Bl_k$ by the scalar $k^2(p^2+pd)$. Similar results for the category of {\em finite-dimensional} $\D$-modules were obtained in \cite{Ger00}.
\end{rem}

We introduce the shorthand notations for various modules in $\Bl_k\, (k\ge 1)$, and
we often omit the index $k$ when there is no confusion:
\begin{align}
  \begin{split}
\TT{n}^{\pm\pm\pm}
=\TT{k;n}^{\pm\pm\pm} &=\TT{f_{k;n}^{\pm\pm\pm}} \; (n\in\Z \backslash \{0,-kp, kd\}),
\\
\TT{0}^{\circ\pm\pm} =\TT{k;0}^{\circ\pm\pm} =\TT{f_{k;0}^{\circ\pm\pm}},
\quad
\TT{-kp}^{\pm\circ\pm}
&=\TT{k;-kp}^{\pm\circ\pm} =\TT{f_{k;-kp}^{\pm\circ\pm}}, \quad
\TT{kd}^{\pm\pm\circ} =\TT{k;kd}^{\pm\pm\circ} =\TT{f_{k;kd}^{\pm\pm\circ}}.
  \end{split}
\end{align}
Similar self-explanatory notations, with $P/M/L$ replacing $T$ above,
will be used for the projective, Verma, and simple modules.

\subsection{The principal block $\Bl_0$ $(\ka \in \Q)$}

The weight poset  $\Wt_0$ is identical to the weight poset  for the principal block with $\ka \not\in\Q$; see Section \ref{sec:irrational}.
The Verma flag structures of tilting modules and projective modules in $\Bl_0$
are exactly the same as for $\Bl_0$ with $\ka \not\in\Q$; see Section \ref{sec:irrational}.
They can be proved in the same way as before, and we will not repeat it here.
These observations motivate us to formulate the following conjecture.

\begin{conj}
 \label{block0:same}
The principal blocks $\Bl_0$ for all parameters $\ka$ are equivalent (as highest weight categories).
\end{conj}
The blocks $\Bl_k$, for distinct $k$, are not equivalent since their weight posets  $\Wt_k$ are pairwise non-isomorphic. But one may ask if the blocks $\Bl_k$ can possibly be derived equivalent.

\subsection{Tilting modules in the blocks $\Bl_k$ for $kp\ge 2,  kd \ge 2$}
\label{subsec:tilting:rational}

In the remainder of Section~\ref{sec:rational}, we make the following assumption
\begin{align}
kp\ge 2, \quad kd \ge 2.
\end{align}
The remaining cases will be treated in Sections~\ref{sec:block1} and ~\ref{sec:block:ka=1}.

We shall describe the Verma flags for all tilting modules in all blocks $\Bl_k$ $(k\ge 1)$ in
Theorems~\ref{thm:tilting:reg}--\ref{thm:tilting:+} below.
The translation functors below from a typical block
to the atypical block $\Bl_k$ are obtained from tensoring with the adjoint module.
Our strategy here is the same as described at the beginning of Section \ref{subsec:tilting:irr}.

We read one fixed sign for the indices at a time for the formulae in Theorem~\ref{thm:tilting:reg}  below
for tilting modules with regular Verma flags.
\begin{thm}
  \label{thm:tilting:reg}
  Assume that $kp\ge 2,  kd \ge 2.$
\\
(1) We have the following Verma flag formulae for $\TT{-n}^{-\pm\pm}$  $(n\ge 1, n\not =kp, kp-1)$
and  for $\TT{n}^{-\pm\pm}$  $(n\ge 1, n \not=kd, kd-1)$ in the block $\Bl_k$:
\begin{align*}
\TT{\pm n}^{---}  =
& \M{\pm n}^{---}   + \M{\pm (n+1)}^{---},
\\
\TT{\pm n}^{--+}  =
& \M{\pm n}^{--+} +\M{\pm n}^{---}  +\M{\pm(n+1)}^{--+}   + \M{\pm(n+1)}^{---},
\\
\TT{\pm n}^{-+-}  =
& \M{\pm n}^{-+-} +\M{\pm n}^{---}  +\M{\pm(n+1)}^{-+-}   + \M{\pm(n+1)}^{---},
\\
\TT{\pm n}^{-++}  =
& \M{\pm n}^{-++} +\M{\pm n}^{-+-} + \M{\pm n}^{--+}  +\M{\pm n}^{---}
 +\M{\pm(n+1)}^{-++}    +\M{\pm(n+1)}^{-+-}   +\M{\pm(n+1)}^{--+}   + \M{\pm(n+1)}^{---}.
\end{align*}

(2) We have the following Verma flag formulae for $\TT{-n}^{+\pm\pm}$  $(n\ge 2, n\not =kp, kp+1)$
and  for $\TT{n}^{+\pm\pm}$  $(n\ge 2, n \not=kd, kd+1)$ in the block $\Bl_k$:
\begin{align*}
\TT{\pm n}^{+--}  =
& \M{\pm n}^{+--} +\M{\pm(n-1)}^{+--}  +\M{\pm(n-1)}^{---}   + \M{\pm n}^{---},
 \\ %
\TT{\pm n}^{+-+}  =
& \M{\pm n}^{+-+} +\M{\pm n}^{+--} + \M{\pm(n-1)}^{+-+}  +\M{\pm(n-1)}^{+--}
 +\M{\pm(n-1)}^{--+}    +\M{\pm(n-1)}^{---}   +\M{\pm n}^{--+}   + \M{\pm n}^{---},
  \\ %
\TT{\pm n}^{++-}  =
& \M{\pm n}^{++-} +\M{\pm n}^{+--} + \M{\pm(n-1)}^{++-}  +\M{\pm(n-1)}^{+--}
 +\M{\pm(n-1)}^{-+-}    +\M{\pm(n-1)}^{---}   +\M{\pm n}^{-+-}   + \M{\pm n}^{---},
  \\ %
\TT{\pm n}^{+++}  =
& \M{\pm n}^{+++} +\M{\pm n}^{++-} +\M{\pm n}^{+-+} +\M{\pm n}^{+--}
+ \M{\pm(n-1)}^{+++}  +\M{\pm(n-1)}^{++-}   +\M{\pm(n-1)}^{+-+}    +\M{\pm(n-1)}^{+--}
\\
&  + \M{\pm(n-1)}^{-++}  +\M{\pm(n-1)}^{-+-}   +\M{\pm(n-1)}^{--+}    +\M{\pm(n-1)}^{---}
+\M{\pm n}^{-++}   + \M{\pm n}^{-+-}+\M{\pm n}^{--+}   + \M{\pm n}^{---}.
\end{align*}
\end{thm}

\begin{proof}
The proof is a rerun of the proofs of Theorems~\ref{thm:tilting1} and \ref{thm:tilting2}, and so let us be brief.

We apply translation functors  $\E$ to $T =\TT{f_{\pm n}^{abc}-(2,0,0)}$ to obtain $\TT{\pm}^{abc}$, for all 8 choices $a,b,c \in \{\pm\}$.
We first check the resulting modules $\E T$ indeed have Verma flags as given in the theorem.
To show these 8 modules $\E  T$ obtained via translation functors (applied on tilting modules) are tilting modules, it remains to show they are indecomposable. We proceed following the anti-dominant order as listed in the theorem, as the proof of the latter formulae assumes the validity of preceding formulae.

We check the lowest term $M_{low}$ on the RHS of the 8 formulae in the theorem is indeed in a Verma flag of the LHS.
Case~1 in (1) is easy by Lemma~\ref{lem:oddhom}.
For Case~ 2 in (1), we have $(\TT{\pm n}^{--+}: \M{\pm(n+1)}^{---}) =[ \M{\pm(n+1)}^{+++} : \LL{\pm n}^{++-}]>0$, since
there is a singular vector of the form $f_2^* u^+$ (for some positive power $*$) where $u^+$ is a singular vector
of $\rho$-shifted weight $f_{{\pm n}}^{+++}$ by Lemma~\ref{lem:oddhom}. Case~3 in (1) is similar.
For Case~ 4 in (1), we have $(\TT{\pm n}^{-++}: \M{\pm(n+1)}^{---}) =[ \M{\pm(n+1)}^{+++} : \LL{\pm n}^{+--}]>0$, since
there is a singular vector of the form $f_1^{**}f_2^* u^+$ (for some positive powers $*, **$).
For example, for Case~4 of (2), we have $(\TT{\pm n}^{+++}: \M{\pm n}^{---}) =[ \M{\pm n}^{+++} : \LL{\pm n}^{---}]>0$, since
there is a singular vector of the form $f_1^{**} f_2^* u^+$ (for some positive powers $*, **$),
where $u^+$ is a singular vector of $\rho$-shifted weight $f_{\pm n}^{-++}$ by Lemma~\ref{lem:evenhom}.
The first 3 cases of (2) are similar and easier.

Finally, by base-by-case inspection we check the tilting module whose highest term is any Verma on the RHS (of any of the 8 formulae in the theorem)
other than the first one has its lowest term being $M_{low}$ or a term not on the RHS.
So it is impossible for such a tilting module to be a direct summand of the modules $\E T$. This proves the desired indecomposability.
\end{proof}

It remains to describe the  irregular tilting modules
$\TT{0}^{\circ\pm\pm}$,
$\TT{1-kp}^{-\pm\pm}$, $\TT{kd-1}^{-\pm\pm}$,
$\TT{-1}^{+\pm\pm}$, $\TT{1}^{+\pm\pm}$,
$\TT{-kp}^{\pm\circ\pm}$,  $\TT{kd}^{\pm\pm\circ}$, $\TT{-1-kp}^{+\pm\pm}$ and $\TT{kd+1}^{+\pm\pm}$.

\begin{thm}
  \label{thm:tilting:0-}
Assume that $kp\ge 2,  kd \ge 2.$
We have the following Verma flags for tilting modules in the block $\Bl_k$:
\begin{align}
\TT{0}^{\circ--}  & =\M{0}^{\circ--} +\M{-1}^{---} +\M{1}^{---},    \notag
\\
\TT{0}^{\circ-+}  & =\M{0}^{\circ-+} +\M{0}^{\circ--} +\M{-1}^{--+} +\M{-1}^{---}  +\M{1}^{--+}  +\M{1}^{---},    \notag
\\
\TT{0}^{\circ+-}  & =\M{0}^{\circ+-} +\M{0}^{\circ--} +\M{1}^{-+-} +\M{-1}^{-+-} +\M{-1}^{---} +\M{1}^{---},     \label{T0}%
\\
\TT{0}^{\circ++}  & =\M{0}^{\circ++} +\M{0}^{\circ+-} +\M{0}^{\circ-+} +\M{0}^{\circ--}
+\M{1}^{-++} +\M{1}^{-+-} +\M{-1}^{-++} +\M{-1}^{-+-}    \notag \\
    &\qquad +\M{-1}^{--+} +\M{-1}^{---}  +\M{1}^{--+} +\M{1}^{---},    \notag
\end{align}
\begin{align}
\TT{1-kp}^{---}  & =\M{1-kp}^{---}  +\M{-kp}^{-\circ-}  +\M{-1-kp}^{---},    \notag
\\
\TT{1-kp}^{--+} & =\M{1-kp}^{--+}  +\M{1-kp}^{---} +\M{-kp}^{-\circ+} +\M{-kp}^{-\circ-} +\M{-1-kp}^{--+}   +\M{-1-kp}^{---},    \notag
\\
\TT{1-kp}^{-+-} & =\M{1-kp}^{-+-}  +\M{1-kp}^{---} +\ 2 \M{-kp}^{-\circ-} +\M{-1-kp}^{-+-}   +\M{-1-kp}^{---},   \label{T1-kp}%
\\
\TT{1-kp}^{-++}  &=\M{1-kp}^{-++}  +\M{1-kp}^{-+-}  +\M{1-kp}^{--+}  +\M{1-kp}^{---}  +2\M{-kp}^{-\circ+} + 2 \M{-kp}^{-\circ-}    \notag\\
&\qquad +\M{-1-kp}^{-++} +\M{-1-kp}^{-+-} +\M{-1-kp}^{--+}   +\M{-1-kp}^{---},     \notag
\\
   \notag\\
\TT{kd-1}^{---}  & =\M{kd-1}^{---}  +\M{kd}^{--\circ}  +\M{kd+1}^{---},    \notag
\\
\TT{kd-1}^{--+} & =\M{kd-1}^{--+}  +\M{kd-1}^{---} +2\M{kd}^{--\circ} +\M{kd+1}^{--+}   +\M{kd+1}^{---},    \notag
\\
\TT{kd-1}^{-+-} & =\M{kd-1}^{-+-}  +\M{kd-1}^{---} + \M{kd}^{-+\circ} + \M{kd}^{--\circ} +\M{kd+1}^{-+-}   +\M{kd+1}^{---},   \label{Tkd-1}
\\
\TT{kd-1}^{-++}  &=\M{kd-1}^{-++}  +\M{kd-1}^{-+-}  +\M{kd-1}^{--+}  +\M{kd-1}^{---}  +2\M{kd}^{-+\circ} + 2 \M{kd}^{--\circ}    \notag\\
&\qquad +\M{kd+1}^{-++} +\M{kd+1}^{-+-} +\M{kd+1}^{--+}   +\M{kd+1}^{---}.    \notag
\end{align}
\end{thm}

\begin{thm}
  \label{thm:tilting:pm1}
Assume that $kp\ge 2,  kd \ge 2.$ We have the following Verma flags for tilting modules in the block $\Bl_k$:
\begin{align}
\TT{-1}^{+--}   &= \M{-1}^{+--}  +\M{0}^{\circ--} + \M{-1}^{---},    \notag
\\
\TT{-1}^{+-+}  &= \M{-1}^{+-+}  +\M{-1}^{+--}  +\M{0}^{\circ-+} +\M{0}^{\circ--} + \M{-1}^{--+} + \M{-1}^{---},    \notag
\\
\TT{-1}^{++-}  &= \M{-1}^{++-}  +\M{-1}^{+--}  +\M{0}^{\circ+-} +\M{0}^{\circ--} + \M{-1}^{-+-} + \M{-1}^{---},   \label{T-1} %
\\
\TT{-1}^{+++}  &= \M{-1}^{+++}  +\M{-1}^{++-}  +\M{-1}^{+-+}  +\M{-1}^{+--}
+\M{0}^{\circ++} +\M{0}^{\circ+-} +\M{0}^{\circ-+} +\M{0}^{\circ--}     \notag \\
    &\qquad + \M{-1}^{-++} + \M{-1}^{-+-}+ \M{-1}^{--+} + \M{-1}^{---},    \notag
\end{align}
\begin{align}
\TT{1}^{+--}   &= \M{1}^{+--}  +\M{0}^{\circ--} + \M{1}^{---},    \notag
\\
\TT{1}^{+-+}  &= \M{1}^{+-+}  +\M{1}^{+--}  +\M{0}^{\circ-+} +\M{0}^{\circ--} + \M{1}^{--+} + \M{1}^{---},    \notag
\\
\TT{1}^{++-}  &= \M{1}^{++-}  +\M{1}^{+--}  +\M{0}^{\circ+-} +\M{0}^{\circ--} + \M{1}^{-+-} + \M{1}^{---},   \label{T1}
\\
\TT{1}^{+++}  &= \M{1}^{+++}  +\M{1}^{++-}  +\M{1}^{+-+}  +\M{1}^{+--}
+\M{0}^{\circ++} +\M{0}^{\circ+-} +\M{0}^{\circ-+} +\M{0}^{\circ--}      \notag \\
    &\qquad + \M{1}^{-++} + \M{1}^{-+-}+ \M{1}^{--+} + \M{1}^{---},    \notag
\end{align}
\end{thm}

\begin{thm}
  \label{thm:tilting:sing}
  Assume that $kp\ge 2,  kd \ge 2.$ We have the following Verma flags for tilting modules in the block $\Bl_k$:
\begin{align}
\TT{-kp}^{-\circ-}&= \M{-kp}^{-\circ-} +\M{-1-kp}^{-+-} +\M{-1-kp}^{---},      \notag
\\
\TT{-kp}^{-\circ+}&=  \M{-kp}^{-\circ+} + \M{-kp}^{-\circ-} + \M{-1-kp}^{-++} + \M{-1-kp}^{-+-}+ \M{-1-kp}^{--+} + \M{-1-kp}^{---},      \notag
\\
\TT{-kp}^{+\circ-}  &= \M{-kp}^{+\circ-}  +\M{1-kp}^{++-} +\M{1-kp}^{+--}   + \M{1-kp}^{-+-} + \M{1-kp}^{---} +\M{-kp}^{-\circ-},     \label{T-kp}
\\
\TT{-kp}^{+\circ+}  &= \M{-kp}^{+\circ+}  +\M{-kp}^{+\circ-}  +\M{1-kp}^{+++}  +\M{1-kp}^{++-} +\M{1-kp}^{+-+} +\M{1-kp}^{+--} \notag \\
                               &\qquad  + \M{1-kp}^{-++} + \M{1-kp}^{-+-} + \M{1-kp}^{--+}+ \M{1-kp}^{---}  +\M{-kp}^{-\circ+} +\M{-kp}^{-\circ-},    \notag
\\
   \notag\\
\TT{kd}^{--\circ}&=  \M{kd}^{--\circ}  + \M{kd+1}^{--+} + \M{kd+1}^{---},       \notag
\\
\TT{kd}^{-+\circ}&= \M{kd}^{-+\circ}  + \M{kd}^{--\circ}  + \M{kd+1}^{-++} + \M{kd+1}^{-+-} + \M{kd+1}^{--+} + \M{kd+1}^{---},      \notag
\\
\TT{kd}^{+-\circ}  &= \M{kd}^{+-\circ}  +\M{kd-1}^{+-+} +\M{kd-1}^{+--}  + \M{kd-1}^{--+} + \M{kd-1}^{---} +\M{kd}^{--\circ},     \label{Tkd}
\\
\TT{kd}^{++\circ}  &= \M{kd}^{++\circ}  +\M{kd}^{+-\circ}  +\M{kd-1}^{+++}  +\M{kd-1}^{++-}  +\M{kd-1}^{+-+} +\M{kd-1}^{+--}  \notag \\
                               &\qquad  + \M{kd-1}^{-++} + \M{kd-1}^{--+} + \M{kd-1}^{-+-}+ \M{kd-1}^{---} +\M{kd}^{-+\circ} +\M{kd}^{--\circ}.    \notag
\end{align}
\end{thm}

\begin{thm}
  \label{thm:tilting:+}
Assume that $kp\ge 2,  kd \ge 2.$ We have the following Verma flags for tilting modules in the block $\Bl_k$:
\begin{align}
\TT{-1-kp}^{+--}  &= \M{-1-kp}^{+--} +\M{-kp}^{+\circ-}  +\M{1-kp}^{+--}
                               +\M{1-kp}^{---}  +\M{-kp}^{-\circ-}   + \M{-1-kp}^{---},     \notag
\\
\TT{-1-kp}^{+-+}  &= \M{-1-kp}^{+-+} +\M{-1-kp}^{+--} +\M{-kp}^{+\circ+}  +\M{-kp}^{+\circ-}  +\M{1-kp}^{+-+} +\M{1-kp}^{+--}    \notag
                            \\
                            &\qquad   +\M{1-kp}^{--+}  +\M{1-kp}^{---}  +\M{-kp}^{-\circ+} +\M{-kp}^{-\circ-}   + \M{-1-kp}^{--+} + \M{-1-kp}^{---},    \notag
\\
\TT{-1-kp}^{++-}    &=  \M{-1-kp}^{++-} +\M{-1-kp}^{+--} +2\M{-kp}^{+\circ-}  +\M{1-kp}^{++-} +\M{1-kp}^{+--}   \label{T-1-kp}
                             \\
                             &\qquad  +\M{1-kp}^{-+-}  +\M{1-kp}^{---}  +2\M{-kp}^{-\circ-}   + \M{-1-kp}^{-+-} + \M{-1-kp}^{---},    \notag
\\
\TT{-1-kp}^{+++}  &= \M{-1-kp}^{+++} +\M{-1-kp}^{++-} +\M{-1-kp}^{+-+} +\M{-1-kp}^{+--} +2\M{-kp}^{+\circ+}  +2\M{-kp}^{+\circ-}     \notag
                                    \\
                                    & +\M{1-kp}^{+++} +\M{1-kp}^{++-} +\M{1-kp}^{+-+} +\M{1-kp}^{+--} +\M{1-kp}^{-++}  +\M{1-kp}^{-+-}  +\M{1-kp}^{--+}  +\M{1-kp}^{---}     \notag  \\
                            &\qquad    + 2\M{-kp}^{-\circ+} + 2\M{-kp}^{-\circ-} + \M{-1-kp}^{-++} + \M{-1-kp}^{-+-} + \M{-1-kp}^{--+} + \M{-1-kp}^{---},   \notag
\end{align}
\begin{align}
\TT{kd+1}^{+--}  &= \M{kd+1}^{+--} +\M{kd}^{+-\circ}  +\M{kd-1}^{+--}
                               +\M{kd-1}^{---}  +\M{kd}^{--\circ}   + \M{kd+1}^{---},    \notag
\\
\TT{kd+1}^{+-+}  &= \M{kd+1}^{+-+} +\M{kd+1}^{+--} +2\M{kd}^{+-\circ}  +\M{kd-1}^{+-+} +\M{kd-1}^{+--}     \notag
                            \\
                            &\qquad   +\M{kd-1}^{--+}  +\M{kd-1}^{---}  +2\M{kd}^{--\circ}   + \M{kd+1}^{--+} + \M{kd+1}^{---},    \notag
\\
\TT{kd+1}^{++-}    &=  \M{kd+1}^{++-} +\M{kd+1}^{+--} +\M{kd}^{++\circ} +\M{kd}^{+-\circ}  +\M{kd-1}^{++-} +\M{kd-1}^{+--}    \label{Tkd+1}
                             \\
                             &\qquad  +\M{kd-1}^{-+-}  +\M{kd-1}^{---}  +\M{kd}^{-+\circ} +\M{kd}^{--\circ}   + \M{kd+1}^{-+-} + \M{kd+1}^{---},    \notag
\\
\TT{kd+1}^{+++}  &= \M{kd+1}^{+++} +\M{kd+1}^{++-} +\M{kd+1}^{+-+} +\M{kd+1}^{+--} +2\M{kd}^{++\circ}  +2\M{kd}^{+-\circ}     \notag
                                    \\
                                    & +\M{kd-1}^{+++} +\M{kd-1}^{++-} +\M{kd-1}^{+-+} +\M{kd-1}^{+--} +\M{kd-1}^{-++}  +\M{kd-1}^{-+-}  +\M{kd-1}^{--+}  +\M{kd-1}^{---}      \notag \\
                            &\qquad  + 2\M{kd}^{-+\circ} + 2\M{kd}^{--\circ}  + \M{kd+1}^{-++} + \M{kd+1}^{-+-} + \M{kd+1}^{--+} + \M{kd+1}^{---}.   \notag
\end{align}
\end{thm}

\begin{proof}[Proofs of Theorems~\ref{thm:tilting:0-}--\ref{thm:tilting:+}]

The usual strategy to prove that these resulting modules are tilting
still applies here. We only specify the new features and new cases below.

The most tedious part of the proofs is to verify that when applying suitable translation functors to suitable modules the resulting modules
admit the Verma flags as described in the theorems.
To obtain all the tilting modules $\TT{f}$ in Theorems~\ref{thm:tilting:reg},  \ref{thm:tilting:0-}, \ref{thm:tilting:sing} and \ref{thm:tilting:+},
we apply translation functors to the initial tilting modules $\TT{f-(2,0,0)}$.
The  4 formulae for tilting modules $\TT{f}$ in \eqref{T-1} in Theorem~\ref{thm:tilting:pm1} require
applying translation functors to the initial tilting modules $\TT{f-(1,1,1)}$, $\TT{f-(1,1,-1)}$, $\TT{f-(1,-1,1)}$ and $\TT{f-(1,-1,-1)}$, respectively; that is,
\begin{align*}
\TT{0,-kp,-2-kd} =\M{0,-kp,-2-kd} &\leadsto\leadsto \TT{-1}^{+--},
\\
\TT{0,-kp,2+kd} =\M{0,-kp,2+kd} +\M{0,-kp,-2-kd} &\leadsto\leadsto \TT{-1}^{+-+},
\\
\TT{0,kp,-2-kd} =\M{0,kp,-2-kd} +\M{0,-kp,-2-kd} &\leadsto\leadsto \TT{-1}^{++-},
\\
\TT{0,kp,2+kd} =\M{0,kp,2+kd} +\M{0,kp,-2-kd} +\M{0,-kp,2+kd} +\M{0,-kp,-2-kd}
&\leadsto\leadsto
\TT{-1}^{+++}.
\end{align*}
The 4 formulae in \eqref{T1} are obtained similarly.

In most of the cases it is straightforward to verify the appearance of lowest terms in the tilting formulae and indecomposability. However, it takes a little extra work to verify that the lowest term $\M{-1-kp}^{---}$ indeed appears in the tilting module in Case~ 4 of \eqref{T1-kp}.
(Formulae \eqref{Tkd-1} can be treated in a way entirely similar to \eqref{T1-kp}.) To this end we make use of Lemma~\ref{tilting:comp}:
We apply a translation functor $\E$ to $\TT{-1-kp, 1, k(p+d)-1}$ to obtain a module $\E \TT{-1-kp, 1, k(p+d)-1}$
with highest term $\M{1-kp, 1, k(p+d)-1}$.
Note, since $(-1-kp, 1, k(p+d)-1)$ is typical we have short exact sequences
\begin{align}
   \label{SES:N1}
& \TT{-1-kp, 1, k(p+d)-1}  \longrightarrow \TT{-1-kp, -1, k(p+d)-1} \longrightarrow 0,
\\
   \label{SES:N2}
& \TT{-1-kp, 1, k(p+d)-1} \longrightarrow \TT{-1-kp, 1, 1-k(p+d)} \longrightarrow 0.
\end{align}
Applying $\E$ to \eqref{SES:N1} we see that $\E  \TT{-1-kp, 1, k(p+d)-1}$ has an indecomposable quotient module
$\E \TT{-1-kp, -1, k(p+d)-1} =\TT{1-kp, -1, k(p+d)-1} =\TT{1-kp}^{--+}$ (as given in Case~2 of \eqref{T1-kp}).
Applying $\E$ to \eqref{SES:N2} we see that $\E  \TT{-1-kp, 1, k(p+d)-1}$ has an indecomposable quotient
$\E \TT{-1-kp, 1, 1-k(p+d)} =\TT{1-kp, 1, 1-k(p+d)} =\TT{1-kp}^{-+-}$  (as given in Case~3 of \eqref{T1-kp}).
It is easy to see that $\M{kd-1}^{-+-}$ and $\M{kd-1}^{--+}$ must appear in $\TT{1-kp}^{-++}$.
Since $\M{kd-1}^{-+-}$ and $\M{kd-1}^{--+}$ are leading terms of $\TT{1-kp}^{--+}$ and $\TT{1-kp}^{-+-}$,
all the terms in  Cases~2--3 of \eqref{T1-kp} will appear in $\TT{1-kp}^{-++}$. Finally, it is clear that the remaining terms
in \eqref{T1-kp},
$\M{kd}^{-+\circ}$ and/or $\M{kd+1}^{-++}$, cannot form a tilting module and hence cannot be a direct summand
of $\E  \TT{-1-kp, 1, k(p+d)-1}$.
We conclude that $\E  \TT{-1-kp, 1, k(p+d)-1}$ is indecomposable and hence must be $\TT{1-kp}^{-++}$.

Finally, the formulae in the Cases 3 and 4 of \eqref{T-1-kp} and in the Cases 2 and 4 of \eqref{Tkd+1} in Theorem \ref{thm:tilting:+} can be established using an analogous argument based on Lemma~\ref{tilting:comp}.
\end{proof}

\subsection{Characters of projectives in $\Bl_k$} 
\label{subsec:proj:rational}

The formulae of Verma flags for tilting modules in
Theorems~\ref{thm:tilting:reg}--\ref{thm:tilting:+}
are readily translated into formulae of Verma flags for projective modules using the identity
\eqref{tilt=proj}. We formulate the results in Propositions~\ref{proj:reg}--\ref{proj:+} below.
Having these explicit formulae available are helpful in identifying the projective tilting modules and then computing the composition factors in Verma modules in the next subsections.

\begin{prop}
  \label{proj:reg}
   Assume that $kp\ge 2,  kd \ge 2.$
\\
(1) We have the following Verma flag formulae for the projective modules $\PP{-n}^{+\pm\pm}$  $(n\ge 1, n\not =kp, kp-1)$
and  for $\PP{n}^{+\pm\pm}$  $(n\ge 1, n \not=kd, kd-1)$ in the block $\Bl_k$:
\begin{align*}
\PP{\pm n}^{+++}  = & \M{\pm n}^{+++}   + \M{\pm (n+1)}^{+++},
\\
\PP{\pm n}^{++-}  = & \M{\pm n}^{++-} +\M{\pm n}^{+++}  +\M{\pm(n+1)}^{++-}   + \M{\pm(n+1)}^{+++},
\\
\PP{\pm n}^{+-+}  =& \M{\pm n}^{+-+} +\M{\pm n}^{+++}  +\M{\pm(n+1)}^{+-+}   + \M{\pm(n+1)}^{+++},
\\
\PP{\pm n}^{+--}  =& \M{\pm n}^{+--} +\M{\pm n}^{+-+} + \M{\pm n}^{++-}  +\M{\pm n}^{+++}
                           +\M{\pm(n+1)}^{+--}+\M{\pm(n+1)}^{+-+}   +\M{\pm(n+1)}^{++-}   + \M{\pm(n+1)}^{+++}.
\end{align*}

(2) We have the following Verma flag formulae for the projective modules $\PP{-n}^{-\pm\pm}$  $(n\ge 2, n\not =kp, kp+1)$
and  for $\PP{n}^{-\pm\pm}$  $(n\ge 2, n \not=kd, kd+1)$ in the block $\Bl_k$:
\begin{align*}
\PP{\pm n}^{-++}  =& \M{\pm n}^{-++} +\M{\pm(n-1)}^{-++}  +\M{\pm(n-1)}^{+++}   + \M{\pm n}^{+++},
\notag \\ %
\PP{\pm n}^{-+-}  =& \M{\pm n}^{-+-} +\M{\pm n}^{-++} + \M{\pm(n-1)}^{-+-}  +\M{\pm(n-1)}^{-++}
 +\M{\pm(n-1)}^{++-}    +\M{\pm(n-1)}^{+++}   +\M{\pm n}^{++-}   + \M{\pm n}^{+++},
\notag \\ %
\PP{\pm n}^{--+}  =& \M{\pm n}^{--+} +\M{\pm n}^{-++} + \M{\pm(n-1)}^{--+}  +\M{\pm(n-1)}^{-++}
 +\M{\pm(n-1)}^{+-+}    +\M{\pm(n-1)}^{+++}   +\M{\pm n}^{+-+}   + \M{\pm n}^{+++},
\notag \\ %
\PP{\pm n}^{---}  =& \M{\pm n}^{---} +\M{\pm n}^{--+} +\M{\pm n}^{-+-} +\M{\pm n}^{-++}
+ \M{\pm(n-1)}^{---}  +\M{\pm(n-1)}^{--+}   +\M{\pm(n-1)}^{-+-}    +\M{\pm(n-1)}^{-++}
\notag  
\\
&  + \M{\pm(n-1)}^{+--}  +\M{\pm(n-1)}^{+-+}   +\M{\pm(n-1)}^{++-}    +\M{\pm(n-1)}^{+++}
+\M{\pm n}^{+--}   + \M{\pm n}^{+-+}+\M{\pm n}^{++-}   + \M{\pm n}^{+++}.
\notag
\end{align*}
\end{prop}

\begin{prop}
  \label{proj:0-}
Assume that $kp\ge 2,  kd \ge 2.$
We have the following Verma flags for projective modules in the block $\Bl_k$:
  \begin{align}
\PP{0}^{\circ++}  & =\M{0}^{\circ++} +\M{-1}^{+++} +\M{1}^{+++},
\notag
\\
\PP{0}^{\circ+-}  & =\M{0}^{\circ+-} +\M{0}^{\circ++} +\M{-1}^{++-}  +\M{-1}^{+++} +\M{1}^{++-}  +\M{1}^{+++},
\notag
\\
\PP{0}^{\circ-+}  & =\M{0}^{\circ-+} +\M{0}^{\circ++} +\M{1}^{+-+} +\M{-1}^{+-+} +\M{-1}^{+++} +\M{1}^{+++},   \label{P0}%
\\
\PP{0}^{\circ--}  & =\M{0}^{\circ--} +\M{0}^{\circ-+} +\M{0}^{\circ+-} +\M{0}^{\circ++}
+\M{1}^{+--} +\M{1}^{+-+} +\M{-1}^{+--} +\M{-1}^{+-+}   \notag \\
    &\quad +\M{-1}^{++-}  +\M{-1}^{+++} +\M{1}^{++-} +\M{1}^{+++},
    \notag
\end{align}
\begin{align}
\PP{1-kp}^{+++}  & =\M{1-kp}^{+++}  +\M{-kp}^{+\circ+}  +\M{-1-kp}^{+++},    \notag
\\
\PP{1-kp}^{++-} & =\M{1-kp}^{++-}  +\M{1-kp}^{+++} +\M{-kp}^{+\circ-} +\M{-kp}^{+\circ+} +\M{-1-kp}^{++-}   +\M{-1-kp}^{+++},    \notag
\\
\PP{1-kp}^{+-+} & =\M{1-kp}^{+-+}  +\M{1-kp}^{+++} +\ 2 \M{-kp}^{+\circ+} +\M{-1-kp}^{+-+}   +\M{-1-kp}^{+++},   \label{P1-kp}%
\\
\PP{1-kp}^{+--}  &=\M{1-kp}^{+--}  +\M{1-kp}^{+-+}  +\M{1-kp}^{++-}  +\M{1-kp}^{+++}  +2\M{-kp}^{+\circ-} + 2 \M{-kp}^{+\circ+}    \notag\\
&\qquad +\M{-1-kp}^{+--} +\M{-1-kp}^{+-+} +\M{-1-kp}^{++-}   +\M{-1-kp}^{+++},     \notag
\\
   \notag\\
\PP{kd-1}^{+++}  & =\M{kd-1}^{+++}  +\M{kd}^{++\circ}  +\M{kd+1}^{+++},    \notag
\\
\PP{kd-1}^{++-} & =\M{kd-1}^{++-}  +\M{kd-1}^{+++} +2\M{kd}^{++\circ} +\M{kd+1}^{++-}   +\M{kd+1}^{+++},    \notag
\\
\PP{kd-1}^{+-+} & =\M{kd-1}^{+-+}  +\M{kd-1}^{+++} + \M{kd}^{+-\circ} + \M{kd}^{++\circ} +\M{kd+1}^{+-+}   +\M{kd+1}^{+++},   \label{Pkd-1}
\\
\PP{kd-1}^{+--}  &=\M{kd-1}^{+--}  +\M{kd-1}^{+-+}  +\M{kd-1}^{++-}  +\M{kd-1}^{+++}  +2\M{kd}^{+-\circ} + 2 \M{kd}^{++\circ}    \notag\\
&\qquad +\M{kd+1}^{+--} +\M{kd+1}^{+-+} +\M{kd+1}^{++-}   +\M{kd+1}^{+++}.    \notag
\end{align}
\end{prop}

\begin{prop}
  \label{proj:pm1}
Assume that $kp\ge 2,  kd \ge 2.$ We have the following Verma flags for projective modules in the block $\Bl_k$
(here we read a fixed sign at a time):
\begin{align}
\PP{\pm 1}^{-++}   &= \M{\pm 1}^{-++}  +\M{0}^{\circ++} + \M{\pm 1}^{+++},
\notag
\\
\PP{\pm 1}^{-+-}  &= \M{\pm 1}^{-+-}  +\M{\pm 1}^{-++}  +\M{0}^{\circ+-} +\M{0}^{\circ++} + \M{\pm 1}^{++-} + \M{\pm 1}^{+++},
\notag
\\
\PP{\pm 1}^{--+}  &= \M{\pm 1}^{--+}  +\M{\pm 1}^{-++}  +\M{0}^{\circ-+} +\M{0}^{\circ++} + \M{\pm 1}^{+-+} + \M{\pm 1}^{+++}, \label{P1-1}
\\
\PP{\pm 1}^{---}  &= \M{\pm 1}^{---}  +\M{\pm 1}^{--+}  +\M{\pm 1}^{-+-}  +\M{\pm 1}^{-++}
+\M{0}^{\circ--} +\M{0}^{\circ-+} +\M{0}^{\circ+-} +\M{0}^{\circ++}    \notag \\
    &\quad + \M{\pm 1}^{+--} + \M{\pm 1}^{+-+}+ \M{\pm 1}^{++-} + \M{\pm 1}^{+++}.
    \notag
\end{align}
\end{prop}

If we set $n=1$ in Proposition~\ref{proj:reg}(2) and interpret the sum of $\M{\pm(n-1)}^{-bc}$ and $\M{\pm(n-1)}^{+bc}$
in $\PP{\pm n}^{-\pm\pm}$  therein as
a single $\M{0}^{\circ bc}$, then we recover
 the formulae for $\PP{\pm 1}^{-\pm\pm}$ in Proposition~\ref{proj:pm1}.

\begin{prop}
  \label{proj:sing}
  Assume that $kp\ge 2,  kd \ge 2.$ We have the following Verma flags for projective modules in the block $\Bl_k$:
\begin{align}
\PP{-kp}^{+\circ+}&= \M{-kp}^{+\circ+} +\M{-1-kp}^{+-+} +\M{-1-kp}^{+++},      \notag
\\
\PP{-kp}^{+\circ-}&=  \M{-kp}^{+\circ-} + \M{-kp}^{+\circ+} + \M{-1-kp}^{+--} + \M{-1-kp}^{+-+} + \M{-1-kp}^{++-} + \M{-1-kp}^{+++},      \notag
\\
\PP{-kp}^{-\circ+}  &= \M{-kp}^{-\circ+}  +\M{1-kp}^{--+} +\M{1-kp}^{-++}   + \M{1-kp}^{+-+} + \M{1-kp}^{+++} +\M{-kp}^{+\circ+},     \label{P-kp}
\\
\PP{-kp}^{-\circ-}  &= \M{-kp}^{-\circ-}  +\M{-kp}^{-\circ+}  +\M{1-kp}^{---}  +\M{1-kp}^{--+} +\M{1-kp}^{-+-} +\M{1-kp}^{-++} \notag \\
                               &\qquad  + \M{1-kp}^{+--} + \M{1-kp}^{+-+} + \M{1-kp}^{++-}+ \M{1-kp}^{+++}  +\M{-kp}^{+\circ-} +\M{-kp}^{+\circ+},    \notag
\\
   \notag\\
\PP{kd}^{++\circ}&=  \M{kd}^{++\circ}  + \M{kd+1}^{++-} + \M{kd+1}^{+++},       \notag
\\
\PP{kd}^{+-\circ}&= \M{kd}^{+-\circ}  + \M{kd}^{++\circ}  + \M{kd+1}^{+--} + \M{kd+1}^{+-+} + \M{kd+1}^{++-} + \M{kd+1}^{+++},      \notag
\\
\PP{kd}^{-+\circ}  &= \M{kd}^{-+\circ}  +\M{kd-1}^{-+-} +\M{kd-1}^{-++}  + \M{kd-1}^{++-} + \M{kd-1}^{+++} +\M{kd}^{++\circ},     \label{Pkd}
\\
\PP{kd}^{--\circ}  &= \M{kd}^{--\circ}  +\M{kd}^{-+\circ}  +\M{kd-1}^{---}  +\M{kd-1}^{--+} +\M{kd-1}^{-+-} +\M{kd-1}^{-++}  \notag \\
                               &\qquad  + \M{kd-1}^{+--} + \M{kd-1}^{++-} + \M{kd-1}^{+-+}+ \M{kd-1}^{+++} +\M{kd}^{+-\circ} +\M{kd}^{++\circ}.    \notag
\end{align}
\end{prop}

\begin{prop}
  \label{proj:+}
Assume that $kp\ge 2,  kd \ge 2.$ We have the following Verma flags for projective modules in the block $\Bl_k$:
\begin{align}
\PP{-1-kp}^{-++}  &= \M{-1-kp}^{-++} +\M{-kp}^{-\circ+}  +\M{1-kp}^{-++}
                               +\M{1-kp}^{+++}  +\M{-kp}^{+\circ+}   + \M{-1-kp}^{+++},     \notag
\\
\PP{-1-kp}^{-+-}  &= \M{-1-kp}^{-+-} +\M{-1-kp}^{-++} +\M{-kp}^{-\circ-}  +\M{-kp}^{-\circ+}  +\M{1-kp}^{-+-} +\M{1-kp}^{-++}    \notag
                            \\
                            &\qquad   +\M{1-kp}^{++-}  +\M{1-kp}^{+++}  +\M{-kp}^{+\circ-} +\M{-kp}^{+\circ+}   + \M{-1-kp}^{++-} + \M{-1-kp}^{+++},    \notag
\\
\PP{-1-kp}^{--+}    &=  \M{-1-kp}^{--+} +\M{-1-kp}^{-++} +2\M{-kp}^{-\circ+}  +\M{1-kp}^{--+} +\M{1-kp}^{-++}   \label{P-1-kp}
                             \\
                             &\qquad  +\M{1-kp}^{+-+}  +\M{1-kp}^{+++}  +2\M{-kp}^{+\circ+}   + \M{-1-kp}^{+-+} + \M{-1-kp}^{+++},    \notag
\\
\PP{-1-kp}^{---}  &= \M{-1-kp}^{---} +\M{-1-kp}^{--+} +\M{-1-kp}^{-+-} +\M{-1-kp}^{-++} +2\M{-kp}^{-\circ-}  +2\M{-kp}^{-\circ+}     \notag
                                    \\
                          & +\M{1-kp}^{---} +\M{1-kp}^{--+} +\M{1-kp}^{-+-} +\M{1-kp}^{-++} +\M{1-kp}^{+--}  +\M{1-kp}^{+-+}  +\M{1-kp}^{++-}  +\M{1-kp}^{+++}     \notag  \\
                            &\qquad    + 2\M{-kp}^{+\circ-} + 2\M{-kp}^{+\circ+} + \M{-1-kp}^{+--} + \M{-1-kp}^{+-+} + \M{-1-kp}^{++-} + \M{-1-kp}^{+++},   \notag
\end{align}
\begin{align}
\PP{kd+1}^{-++}  &= \M{kd+1}^{-++} +\M{kd}^{-+\circ}  +\M{kd-1}^{-++}
                               +\M{kd-1}^{+++}  +\M{kd}^{++\circ}   + \M{kd+1}^{+++},    \notag
\\
\PP{kd+1}^{-+-}  &= \M{kd+1}^{-+-} +\M{kd+1}^{-++} +2\M{kd}^{-+\circ}  +\M{kd-1}^{-+-} +\M{kd-1}^{-++}     \notag
                            \\
                            &\qquad   +\M{kd-1}^{++-}  +\M{kd-1}^{+++}  +2\M{kd}^{++\circ}   + \M{kd+1}^{++-} + \M{kd+1}^{+++},    \notag
\\
\PP{kd+1}^{--+}    &=  \M{kd+1}^{--+} +\M{kd+1}^{-++} +\M{kd}^{--\circ} +\M{kd}^{-+\circ}  +\M{kd-1}^{--+} +\M{kd-1}^{-++}    \label{Pkd+1}
                             \\
                             &\qquad  +\M{kd-1}^{+-+}  +\M{kd-1}^{+++}  +\M{kd}^{+-\circ} +\M{kd}^{++\circ}   + \M{kd+1}^{+-+} + \M{kd+1}^{+++},    \notag
\\
\PP{kd+1}^{---}  &= \M{kd+1}^{---} +\M{kd+1}^{--+} +\M{kd+1}^{-+-} +\M{kd+1}^{-++} +2\M{kd}^{--\circ}  +2\M{kd}^{-+\circ}     \notag
                                    \\
                          & +\M{kd-1}^{---} +\M{kd-1}^{--+} +\M{kd-1}^{-+-} +\M{kd-1}^{-++} +\M{kd-1}^{+--}  +\M{kd-1}^{+-+}  +\M{kd-1}^{++-}  +\M{kd-1}^{+++}      \notag \\
                            &\qquad  + 2\M{kd}^{+-\circ} + 2\M{kd}^{++\circ}  + \M{kd+1}^{+--} + \M{kd+1}^{+-+} + \M{kd+1}^{++-} + \M{kd+1}^{+++}.   \notag
\end{align}
\end{prop}

\subsection{Projective tilting modules in $\Bl_k$} 
\label{subsec:prinj:kp,kd>1}

We would like to determine when a Verma flag of a tilting module could match a Verma flag of a projective module.

By inspection of the formulae for Verma flags of the tilting modules in $\Bl_k$ in Section \ref{subsec:tilting:rational},
all tilting modules except $\TT{-kp}^{+\circ \mp}$ and $\TT{kd}^{+\mp\circ}$ have lowest terms (in the Verma flags) of the form $M_{n}^{---}$ $(n\neq 0)$.
On the other hand, by the formulae in Section \ref{subsec:proj:rational}, the projective modules $\PP{n}^{---}$  (for $n\neq 0, kd, -kp$) have highest terms (in the Verma flags) of the form $M_{n}^{+++}$.
The tilting modules $\TT{-kp}^{+\circ \mp}$ have lowest terms $\M{-kp}^{-\circ-}$, while
the tilting modules $\TT{kd}^{+\mp\circ}$ have lowest terms $\M{kd}^{--\circ}$.
On the other hand, the projective modules $\PP{-kp}^{-\circ-}$ have highest terms $\M{-kp}^{+\circ+}$, while
the projective modules $\PP{kd}^{--\circ}$ have highest terms $\M{kd}^{++\circ}$.

Hence the only possible matchings of Verma flags are between $\TT{-kp}^{+\circ+}$ and  $\PP{-kp}^{-\circ-}$, $\TT{kd}^{++\circ}$ and $\PP{kd}^{--\circ}$, as well as
$\TT{n}^{+++} $ and $\PP{n}^{---}$ (for $n\neq 0, kd, -kp$),  respectively.
Now, a similar argument as in the proof of Theorem \ref{T=P:irrational} enables us to establish the following.

\begin{thm}
 \label{T=P:rational}
Assume $kp\ge 2,  kd \ge 2$. We have the following isomorphisms between projective and tilting modules in $\Bl_k$:
\[
\TT{-kp}^{+\circ+} \cong \PP{-kp}^{-\circ-},
\quad
\TT{kd}^{++\circ}\cong \PP{kd}^{--\circ},
\quad
\TT{n}^{+++} \cong \PP{n}^{---} \;\; (n\in \Z \backslash \{0, kd, -kp\}).
\]
Furthermore, there are no other projective tilting modules of atypical weights in $\Bl_k$.
\end{thm}

\subsection{Composition factors of Verma modules in $\Bl_k$} 
  \label{subsec:comp:Bk}

Using the BGG reciprocity \eqref{BGG}, we obtain the formulae for composition factors of Verma modules in $\Bl_k$, for $k\ge 1$,
in Propositions~\ref{composition1:r}
--\ref{comp-1-kp:kp>1} below.
The composition factors for Verma modules $\M{\pm 1}^{+\pm\pm}$
turn out to have a uniform description as for $\M{\pm n}^{+\pm\pm}$ with $n\ge 2$,
as given in Proposition~\ref{composition1:r}.

\begin{prop}
\label{composition1:r}
Assume $kp\ge 2, kd\ge 2$.
We have the following composition factors for Verma modules $\M{-n}^{\pm\pm\pm}$  $(n\ge 1, n\not =kp, kp \pm 1)$
and  for $\M{n}^{\pm\pm\pm}$  $(n\ge 1, n \not=kd, kd \pm 1)$ in the block $\Bl_k$:
\begin{align*}
\M{\pm n}^{---} =& \LL{\pm n}^{---} +\LL{\pm(n+1)}^{---},
 \\
\M{\pm n}^{--+} =& \LL{\pm n}^{--+} + \LL{\pm n}^{---} +\LL{\pm(n+1)}^{--+} + \LL{\pm(n+1)}^{---},
 \\
\M{\pm n}^{-+-} =& \LL{\pm n}^{-+-} + \LL{\pm n}^{---} +\LL{\pm(n+1)}^{-+-} + \LL{\pm(n+1)}^{---},
 \\
\M{\pm n}^{-++}  =& \LL{\pm n}^{-++} +\LL{\pm n}^{-+-} + \LL{\pm n}^{--+}  +\LL{\pm n}^{---}
+\LL{\pm(n+1)}^{-++}  +\LL{\pm(n+1)}^{-+-} +\LL{\pm(n+1)}^{--+} + \LL{\pm(n+1)}^{---};
\end{align*}
\begin{align*}
\M{\pm n}^{+--} =& \LL{\pm n}^{+--} +\LL{\pm(n-1)}^{+--}  + \LL{\pm n}^{---} + \LL{\pm(n+1)}^{---},
 \\
\M{\pm n}^{+-+}  =& \LL{\pm n}^{+-+}  +\LL{\pm n}^{+--}  + \LL{\pm(n-1)}^{+-+}    + \LL{\pm(n-1)}^{+--}  +\LL{\pm n}^{--+}    + \LL{\pm n}^{---} + \LL{\pm(n+1)}^{--+} + \LL{\pm(n+1)}^{---},
 \\
\M{\pm n}^{++-}  =& \LL{\pm n}^{++-}   +\LL{\pm n}^{+--}  + \LL{\pm(n-1)}^{++-}   + \LL{\pm(n-1)}^{+--}   +\LL{\pm n}^{-+-}   + \LL{\pm n}^{---} + \LL{\pm(n+1)}^{-+-}   + \LL{\pm(n+1)}^{---},
 \\
\M{\pm n}^{+++}  =& \LL{\pm n}^{+++}  +\LL{\pm n}^{++-} + \LL{\pm n}^{+-+}  +\LL{\pm n}^{+--}
   +\LL{\pm(n-1)}^{+++}   +\LL{\pm(n-1)}^{++-} + \LL{\pm(n-1)}^{+-+}  +\LL{\pm(n-1)}^{+--} &\\
+& \LL{\pm n}^{-++}  +\LL{\pm n}^{-+-}   +\LL{\pm n}^{--+}    + \LL{\pm n}^{---}
 +\LL{\pm(n+1)}^{-++}  + \LL{\pm(n+1)}^{-+-}  +\LL{\pm(n+1)}^{--+}  +\LL{\pm(n+1)}^{---}.
\end{align*}
(In the last four formulae when $n =1$, it is understood that $\LL{\pm(n-1)}^{+**} =\LL{0}^{\circ **}$.)
\end{prop}

\begin{prop}
 \label{composition:0}
 Assume $kp\ge 2, kd\ge 2$.
We have the following composition factors of Verma modules in $\Bl_k$:
\begin{align*}
\M{0}^{\circ--} =& \LL{0}^{\circ--}  + \LL{-1}^{---} + \LL{1}^{---},
 \\
\M{0}^{\circ-+} =& \LL{0}^{\circ-+} +\LL{0}^{\circ--}  + \LL{-1}^{--+} + \LL{1}^{--+}   + \LL{-1}^{---} + \LL{1}^{---},
 \\
\M{0}^{\circ+-} =& \LL{0}^{\circ+-} +\LL{0}^{\circ--}  + \LL{-1}^{-+-}  + \LL{1}^{-+-} + \LL{-1}^{---} + \LL{1}^{---},
 \\
\M{0}^{\circ++} =& \LL{0}^{\circ++} +\LL{0}^{\circ+-}  +\LL{0}^{\circ-+} +\LL{0}^{\circ--}
+ \LL{1}^{-++}  + \LL{1}^{-+-} + \LL{-1}^{-++} + \LL{-1}^{-+-} \\
&\quad + \LL{-1}^{--+}  + \LL{1}^{--+} + \LL{-1}^{---} + \LL{1}^{---}.
\end{align*}
\end{prop}

\begin{prop}
 \label{composition:kd,kp>1}
 Assume $kp\ge 2, kd\ge 2$.
We have the following composition factors of Verma modules in $\Bl_k$:
\begin{align*}
\M{kd}^{--\circ} &=\LL{kd}^{--\circ} +\LL{kd+1}^{--+} + 2\LL{kd+1}^{---},
\\
\M{kd}^{-+\circ} &=\LL{kd}^{-+\circ} +\LL{kd}^{--\circ} +\LL{kd+1}^{-++} + 2\LL{kd+1}^{-+-} +\LL{kd+1}^{--+} + 2\LL{kd+1}^{---},
\\
\M{kd}^{+-\circ} &=\LL{kd}^{+-\circ} +\LL{kd-1}^{+-+} + 2\LL{kd-1}^{+--} +\LL{kd}^{--\circ} +\LL{kd+1}^{--+} + 2\LL{kd+1}^{---},
\\
\M{kd}^{++\circ} &=\LL{kd}^{++\circ} +\LL{kd}^{+-\circ}
+\LL{kd-1}^{+++} + 2\LL{kd-1}^{++-} +\LL{kd-1}^{+-+} + 2\LL{kd-1}^{+--}
\\
&\qquad\quad +\LL{kd}^{-+\circ} +\LL{kd}^{--\circ}
+\LL{kd+1}^{-++} + 2\LL{kd+1}^{-+-} +\LL{kd+1}^{--+} + 2\LL{kd+1}^{---};
\\ \\
%
\M{-kp}^{-\circ-} &=\LL{-kp}^{-\circ-} +\LL{-1-kp}^{-+-} + 2\LL{-1-kp}^{---},
\\
\M{-kp}^{-\circ+} &=\LL{-kp}^{-\circ+} +\LL{-kp}^{-\circ-} +\LL{-1-kp}^{-++} + \LL{-1-kp}^{-+-} +2\LL{-1-kp}^{--+} + 2\LL{-1-kp}^{---},
\\
\M{-kp}^{+\circ-} &=\LL{-kp}^{+\circ-} +\LL{1-kp}^{++-} + 2\LL{1-kp}^{+--} +\LL{-kp}^{-\circ-} +\LL{-1-kp}^{-+-} + 2\LL{-1-kp}^{---},
\\
\M{-kp}^{+\circ+} &=\LL{-kp}^{+\circ+} +\LL{-kp}^{+\circ-}
+\LL{1-kp}^{+++} + \LL{1-kp}^{++-} +2\LL{1-kp}^{+-+} + 2\LL{1-kp}^{+--}
\\
&\qquad\quad +\LL{-kp}^{-\circ+} +\LL{-kp}^{-\circ-}
+\LL{-1-kp}^{-++} + \LL{-1-kp}^{-+-} +2\LL{-1-kp}^{--+} + 2\LL{-1-kp}^{---}.
\end{align*}\end{prop}

\begin{prop}
 \label{composition:-1}
 Assume $kp\ge 2, kd\ge 2$.
We have the following composition factors of Verma modules in $\Bl_k$:
\begin{align*}
\M{kd-1}^{---} &=\LL{kd-1}^{---} +\LL{kd}^{--\circ} + \LL{kd+1}^{---},
\\
\M{kd-1}^{--+} &=\LL{kd-1}^{--+} +\LL{kd-1}^{---} +\LL{kd}^{--\circ} +\LL{kd+1}^{--+} +\LL{kd+1}^{---},
\\
\M{kd-1}^{-+-} &=\LL{kd-1}^{-+-} +\LL{kd-1}^{---} +\LL{kd}^{-+\circ} + \LL{kd}^{--\circ} +\LL{kd+1}^{-+-} + \LL{kd+1}^{---},
\\
\M{kd-1}^{-++} &=\LL{kd-1}^{-++} +\LL{kd-1}^{-+-} +\LL{kd-1}^{--+} +\LL{kd-1}^{---} +
\LL{kd}^{-+\circ} +\LL{kd}^{--\circ} +
\LL{kd+1}^{-++} +\LL{kd+1}^{-+-}+ \LL{kd+1}^{--+} +\LL{kd+1}^{---};
\\
\\
\M{kd-1}^{+--} &=\LL{kd-1}^{+--} +\LL{kd-2}^{+--} +\LL{kd-1}^{---}+\LL{kd}^{--\circ} + \LL{kd+1}^{---},
\\
\M{kd-1}^{+-+} &=\LL{kd-1}^{+-+} +\LL{kd-1}^{+--} +\LL{kd-2}^{+-+} +\LL{kd-2}^{+--} +\LL{kd-1}^{--+} +\LL{kd-1}^{---} +\LL{kd}^{--\circ} +\LL{kd+1}^{--+} +\LL{kd+1}^{---},
\\
\M{kd-1}^{++-} &=\LL{kd-1}^{++-} +\LL{kd-1}^{+--} +\LL{kd-2}^{++-} +\LL{kd-2}^{+--} +\LL{kd-1}^{-+-} +\LL{kd-1}^{---} +\LL{kd}^{-+\circ} +\LL{kd}^{--\circ} +\LL{kd+1}^{-+-} +\LL{kd+1}^{---},
\\
\M{kd-1}^{+++} &= \LL{kd-1}^{+++} +\LL{kd-1}^{++-} +\LL{kd-1}^{+-+} +\LL{kd-1}^{+--}
+ \LL{kd-2}^{+++} +\LL{kd-2}^{++-} +\LL{kd-2}^{+-+} +\LL{kd-2}^{+--}
\\
&\quad
+\LL{kd-1}^{-++} +\LL{kd-1}^{-+-} +\LL{kd-1}^{--+} +\LL{kd-1}^{---}
+\LL{kd}^{-+\circ} +\LL{kd}^{--\circ}
+\LL{kd+1}^{-++} +\LL{kd+1}^{-+-} +\LL{kd+1}^{--+} +\LL{kd+1}^{---}.
\end{align*}
\end{prop}

\begin{prop}
 \label{composition:+1}
 Assume $kp\ge 2, kd\ge 2$.
We have the following composition factors of Verma modules in $\Bl_k$:
\begin{align*}
\M{kd+1}^{---} &=\LL{kd+1}^{---}  + \LL{kd+2}^{---},
\\
\M{kd+1}^{--+} &=\LL{kd+1}^{--+}  + \LL{kd+1}^{---}  + \LL{kd+2}^{--+} + \LL{kd+2}^{---},
\\
\M{kd+1}^{-+-} &=\LL{kd+1}^{-+-}  + \LL{kd+1}^{---}  + \LL{kd+2}^{-+-} + \LL{kd+2}^{---},
\\
\M{kd+1}^{-++} &=\LL{kd+1}^{-++}  + \LL{kd+1}^{-+-}  + \LL{kd+1}^{--+}  + \LL{kd+1}^{---}  +
\LL{kd+2}^{-++} + \LL{kd+2}^{-+-} +\LL{kd+2}^{--+} + \LL{kd+2}^{---};
\\
\\
\M{kd+1}^{+--} &=\LL{kd+1}^{+--} + \LL{kd}^{+-\circ}  + \LL{kd-1}^{+--}  + \LL{kd+1}^{---}  + \LL{kd+2}^{---},
\\
\M{kd+1}^{+-+} &=\LL{kd+1}^{+-+} + \LL{kd+1}^{+--} + \LL{kd}^{+-\circ}  + \LL{kd-1}^{+-+} + \LL{kd-1}^{+--}
  + \LL{kd+1}^{--+}  + \LL{kd+1}^{---}  + \LL{kd+2}^{--+} + \LL{kd+2}^{---},
\\
\M{kd+1}^{++-} &=\LL{kd+1}^{++-} + \LL{kd+1}^{+--} + \LL{kd}^{++\circ} + \LL{kd}^{+-\circ}  + \LL{kd-1}^{++-} + \LL{kd-1}^{+--}
  + \LL{kd+1}^{-+-}  + \LL{kd+1}^{---}  + \LL{kd+2}^{-+-} + \LL{kd+2}^{---},
\\
\M{kd+1}^{+++} &=\LL{kd+1}^{+++} + \LL{kd+1}^{++-}  +  \LL{kd+1}^{+-+} + \LL{kd+1}^{+--}
+ \LL{kd}^{++\circ} + \LL{kd}^{+-\circ}
+ \LL{kd-1}^{+++} + \LL{kd-1}^{++-}  +  \LL{kd-1}^{+-+} + \LL{kd-1}^{+--}
\\
&\qquad
  + \LL{kd+1}^{-++}  + \LL{kd+1}^{-+-}    + \LL{kd+1}^{--+}  + \LL{kd+1}^{---}
  + \LL{kd+2}^{-++} + \LL{kd+2}^{-+-}  + \LL{kd+2}^{--+} + \LL{kd+2}^{---}.
\end{align*}
\end{prop}

The formulae in Propositions~\ref{comp:1-kp:kp>1}--\ref{comp-1-kp:kp>1}
below are in a pattern dual to those in Propositions~\ref{composition:-1}--\ref{composition:+1}.

\begin{prop}
 \label{comp:1-kp:kp>1}
 Assume $kp\ge 2, kd\ge 2$.
We have the following composition factors of Verma modules in $\Bl_k$:
\begin{align*}
\M{1-kp}^{---} &=\LL{1-kp}^{---} +\LL{-kp}^{-\circ-} + \LL{-1-kp}^{---},
\\
\M{1-kp}^{--+} &=\LL{1-kp}^{--+} +\LL{1-kp}^{---} +\LL{-kp}^{-\circ+} + \LL{-kp}^{-\circ-} +\LL{-1-kp}^{--+} + \LL{-1-kp}^{---},
\\
\M{1-kp}^{-+-} &=\LL{1-kp}^{-+-} +\LL{1-kp}^{---} +\LL{-kp}^{-\circ-} +\LL{-1-kp}^{-+-} +\LL{-1-kp}^{---},
\\
\M{1-kp}^{-++} &=\LL{1-kp}^{-++} +\LL{1-kp}^{-+-} +\LL{1-kp}^{--+} +\LL{1-kp}^{---} +
\LL{-kp}^{-\circ+} +\LL{-kp}^{-\circ-}
\\
&\qquad
+\LL{-1-kp}^{-++} +\LL{-1-kp}^{-+-}+ \LL{-1-kp}^{--+} +\LL{-1-kp}^{---};
\\  \\
\M{1-kp}^{+--} &=\LL{1-kp}^{+--} +\LL{2-kp}^{+--} +\LL{1-kp}^{---}+\LL{-kp}^{-\circ-} + \LL{-1-kp}^{---},
\\
\M{1-kp}^{+-+} &=\LL{1-kp}^{+-+} +\LL{1-kp}^{+--} +\LL{2-kp}^{+-+} +\LL{2-kp}^{+--} +\LL{1-kp}^{--+} +\LL{1-kp}^{---} +\LL{-kp}^{-\circ+} +\LL{-kp}^{-\circ-} +\LL{-1-kp}^{--+} +\LL{-1-kp}^{---},
\\
\M{1-kp}^{++-} &=\LL{1-kp}^{++-} +\LL{1-kp}^{+--} +\LL{2-kp}^{++-} +\LL{2-kp}^{+--} +\LL{1-kp}^{-+-} +\LL{1-kp}^{---} +\LL{-kp}^{-\circ-} +\LL{-1-kp}^{-+-} +\LL{-1-kp}^{---},
\\
\M{1-kp}^{+++} &= \LL{1-kp}^{+++} +\LL{1-kp}^{++-} +\LL{1-kp}^{+-+} +\LL{1-kp}^{+--}
+ \LL{2-kp}^{+++} +\LL{2-kp}^{++-} +\LL{2-kp}^{+-+} +\LL{2-kp}^{+--}
\\
&\qquad
+\LL{1-kp}^{-++} +\LL{1-kp}^{-+-} +\LL{1-kp}^{--+} +\LL{1-kp}^{---}
+\LL{-kp}^{-\circ+} +\LL{-kp}^{-\circ-}
\\
&\qquad
+\LL{-1-kp}^{-++} +\LL{-1-kp}^{-+-} +\LL{-1-kp}^{--+} +\LL{-1-kp}^{---}.
\end{align*}
\end{prop}

\begin{prop}
 \label{comp-1-kp:kp>1}
 Assume $kp\ge 2, kd\ge 2$.
We have the following composition factors of Verma modules in $\Bl_k$:
\begin{align*}
\M{-1-kp}^{---} &=\LL{-1-kp}^{---}  + \LL{-2-kp}^{---},
\\
\M{-1-kp}^{--+} &=\LL{-1-kp}^{--+}  + \LL{-1-kp}^{---}  + \LL{-2-kp}^{--+} + \LL{-2-kp}^{---},
\\
\M{-1-kp}^{-+-} &=\LL{-1-kp}^{-+-}  + \LL{-1-kp}^{---}  + \LL{-2-kp}^{-+-} + \LL{-2-kp}^{---},
\\
\M{-1-kp}^{-++} &=\LL{-1-kp}^{-++}  + \LL{-1-kp}^{-+-}  + \LL{-1-kp}^{--+}  + \LL{-1-kp}^{---}  +
\LL{-2-kp}^{-++} + \LL{-2-kp}^{-+-} +\LL{-2-kp}^{--+} + \LL{-2-kp}^{---};
\\  \\
\M{-1-kp}^{+--} &=\LL{-1-kp}^{+--} + \LL{-kp}^{+\circ-}  + \LL{1-kp}^{+--}  + \LL{-1-kp}^{---}  + \LL{-2-kp}^{---},
\\
\M{-1-kp}^{+-+} &=\LL{-1-kp}^{+-+} + \LL{-1-kp}^{+--} + \LL{-kp}^{+\circ+} + \LL{-kp}^{+\circ-}  + \LL{1-kp}^{+-+} + \LL{1-kp}^{+--}
 \\
 &\qquad + \LL{-1-kp}^{--+}  + \LL{-1-kp}^{---}  + \LL{-2-kp}^{--+} + \LL{-2-kp}^{---},
\\
\M{-1-kp}^{++-} &=\LL{-1-kp}^{++-} + \LL{-1-kp}^{+--} + \LL{-kp}^{+\circ-}  + \LL{1-kp}^{++-} + \LL{1-kp}^{+--}
  + \LL{-1-kp}^{-+-}  + \LL{-1-kp}^{---}  + \LL{-2-kp}^{-+-} + \LL{-2-kp}^{---},
\\
\M{-1-kp}^{+++} &=\LL{-1-kp}^{+++} + \LL{-1-kp}^{++-}  +  \LL{-1-kp}^{+-+} + \LL{-1-kp}^{+--}
+ \LL{-kp}^{+\circ+} + \LL{-kp}^{+\circ-}
\\
&\qquad + \LL{1-kp}^{+++} + \LL{1-kp}^{++-}  +  \LL{1-kp}^{+-+} + \LL{1-kp}^{+--}
\\
&\qquad
  + \LL{-1-kp}^{-++}  + \LL{-1-kp}^{-+-}    + \LL{-1-kp}^{--+}  + \LL{-1-kp}^{---}
  + \LL{-2-kp}^{-++} + \LL{-2-kp}^{-+-}  + \LL{-2-kp}^{--+} + \LL{-2-kp}^{---}.
\end{align*}
\end{prop}

\section{Character formulae in the block $\Bl_1$, for $\ka \in \Z_{\ge 2}$}
 \label{sec:block1}

The character formulae in Section~\ref{sec:rational} exclude the blocks $\Bl_1$ for $\ka  \in \Z^{>0}$.
In this section we work  out the character formulae for tilting modules and projective modules as well as the composition factors of Verma modules in  the block $\Bl_1$ with $\ka \in \Z_{\ge 2}$.
The block $\Bl_1$ with $\ka=1$ will be treated in Section~\ref{sec:block:ka=1}.

\subsection{Verma flags for tilting modules in $\Bl_1$} 
   \label{subsec:tilting2}

Throughout Sections \ref{subsec:tilting2}--\ref{subsec:comp2}, we assume $\ka =p/d \in \Z_{\ge 2}$, that is,
$p\ge 2,  d=1.$

We shall describe the Verma flags for all tilting modules in the block $\Bl_1$.
The translation functors below from a typical block
to the atypical block $\Bl_k$ are obtained from tensoring with the adjoint module unless otherwise specified.
Our strategy here is the same as described at the beginning of Section \ref{subsec:tilting:irr}.

Theorem~\ref{thm:tilting:reg} for regular tilting modules remains valid for $\ka\in \Z_{\ge 2}$ here.

The remaining irregular tilting modules are:
$\TT{0}^{\circ\pm\pm}$,
$\TT{1-kp}^{-\pm\pm}$,
$\TT{-1}^{+\pm\pm}$,
$\TT{-kp}^{\pm\circ\pm}$,  $\TT{1}^{\pm\pm\circ}$, $\TT{-1-kp}^{+\pm\pm}$ and $\TT{2}^{+\pm\pm}$.
The earlier formulae for irregular tilting modules $\TT{1-kp}^{-\pm\pm}$,
$\TT{-1}^{+\pm\pm}$, $\TT{-kp}^{\pm\circ\pm}$ and $\TT{-1-kp}^{+\pm\pm}$ remain valid.
More precisely, the formula \eqref{T1-kp} for $\TT{1-kp}^{-\pm\pm}$ in Theorem~\ref{thm:tilting:0-},
the formula \eqref{T-1}  for $\TT{-1}^{+\pm\pm}$  in Theorem~\ref{thm:tilting:pm1},
the formula \eqref{T-kp} for $\TT{-kp}^{\pm\circ\pm}$
in Theorem~\ref{thm:tilting:sing},
and the formula \eqref{T-1-kp} for $\TT{-1-kp}^{+\pm\pm}$ in Theorem~\ref{thm:tilting:+}
 remain valid in the current setting. We summarize these as follows for future reference.

 \begin{thm}
   \label{tilting:same}
Assume that  $\ka\in \Z_{\ge 2}$.
The formulae in Theorem~\ref{thm:tilting:reg} for regular tilting modules
and the formulae for irregular tilting modules $\TT{1-kp}^{-\pm\pm}$,
$\TT{-1}^{+\pm\pm}$,  $\TT{-kp}^{\pm\circ\pm}$ and $\TT{-1-kp}^{+\pm\pm}$ in Section \ref{subsec:tilting:rational} remain valid.
 \end{thm}

The new irregular cases
$\TT{0}^{\circ\pm\pm}$,  $\TT{2}^{+\pm\pm}$ and $\TT{1}^{\pm\pm\circ}$, respectively, are treated in
Theorems~\ref{thm:tilting:0:kd=1}, \ref{thm:tilting:2:kd=1} and  \ref{thm:tilting:1:kd=1} below.

\begin{thm}
  \label{thm:tilting:0:kd=1}
Assume that $\ka\in \Z_{\ge 2}$.
We have the following Verma flags for tilting modules in the block $\Bl_1$:
\begin{align}
\TT{0}^{\circ--}  & =\M{0}^{\circ--} +\M{-1}^{---} +\M{1}^{--\circ} +\M{2}^{---},    \notag
\\
\TT{0}^{\circ-+}  & =\M{0}^{\circ-+} +\M{0}^{\circ--} +\M{-1}^{--+} +\M{-1}^{---}
                                +2 \M{1}^{--\circ}  +\M{2}^{--+} +\M{2}^{---},    \notag
\\
\TT{0}^{\circ+-}  & =\M{0}^{\circ+-} +\M{0}^{\circ--} +\M{1}^{-+\circ} +\M{-1}^{-+-} +\M{-1}^{---} +\M{1}^{--\circ}  +\M{2}^{-+-} +\M{2}^{---},     \label{T0:kd=1}%
\\
\TT{0}^{\circ++}  & =\M{0}^{\circ++} +\M{0}^{\circ+-} +\M{0}^{\circ-+} +\M{0}^{\circ--}
+ 2\M{1}^{-+\circ} +\M{-1}^{-++} +\M{-1}^{-+-}   \notag \\
    &\qquad +\M{-1}^{--+}  +\M{-1}^{---}  +2\M{1}^{--\circ}  +\M{2}^{-++} +\M{2}^{-+-} +\M{2}^{--+} +\M{2}^{---}.    \notag
\end{align}
\end{thm}

\begin{thm}
 \label{thm:tilting:2:kd=1}
Assume that $\ka\in \Z_{\ge 2}.$
We have the following Verma flags for tilting modules in the block $\Bl_1$:
\begin{align}
\TT{2}^{+--}  &= \M{2}^{+--} +\M{1}^{+-\circ}  +\M{0}^{\circ--} +\M{1}^{--\circ}   + \M{2}^{---},    \notag
\\
\TT{2}^{+-+}  &= \M{2}^{+-+} +\M{2}^{+--} +2\M{1}^{+-\circ}  +\M{0}^{\circ-+} +\M{0}^{\circ--} +2\M{1}^{--\circ}   + \M{2}^{--+} + \M{2}^{---},    \notag
\\
\TT{2}^{++-}    &=  \M{2}^{++-} +\M{2}^{+--} +\M{1}^{++\circ} +\M{1}^{+-\circ}  +\M{0}^{\circ+-} +\M{0}^{\circ--}      \label{T2:kd=1}
                             \\ &\qquad +\M{1}^{-+\circ} +\M{1}^{--\circ}   + \M{2}^{-+-} + \M{2}^{---},    \notag
\\
\TT{2}^{+++}  &= \M{2}^{+++} +\M{2}^{++-} +\M{2}^{+-+} +\M{2}^{+--} +2\M{1}^{++\circ}  +2\M{1}^{+-\circ}    +\M{0}^{\circ++} +\M{0}^{\circ+-}
\notag
\\
&+\M{0}^{\circ-+} +\M{0}^{\circ--}    + 2\M{1}^{-+\circ} + 2\M{1}^{--\circ}  + \M{2}^{-++} + \M{2}^{-+-} + \M{2}^{--+} + \M{2}^{---}.   \notag
\end{align}
\end{thm}

\begin{proof}[Proof of Theorems~\ref{thm:tilting:0:kd=1} and \ref{thm:tilting:2:kd=1}]
To obtain all the tilting modules $\TT{f}$ in these cases, we apply translation functors to the initial tilting modules $\TT{f-(2,0,0)}$.
The rest is standard and we skip the detail.
\end{proof}

\begin{thm}
   \label{thm:tilting:1:kd=1}
Assume that $\ka\in \Z_{\ge 2}$. We have the following Verma flags for the tilting modules in the block $\Bl_1$:
\begin{align}
\TT{1}^{--\circ}&=  \M{1}^{--\circ}  + \M{2}^{--+} + \M{2}^{---},       \notag
\\
\TT{1}^{-+\circ}&= \M{1}^{-+\circ}  + \M{1}^{--\circ}  + \M{2}^{-++} + \M{2}^{-+-} + \M{2}^{--+} + \M{2}^{---},       \label{T1:kd=1}
\\
\TT{1}^{+-\circ}  & =\M{1}^{+-\circ} +\M{0}^{\circ-+} +\M{0}^{\circ--} +\M{1}^{--\circ},    \notag
\\
\TT{1}^{++\circ}  & =\M{1}^{++\circ} + \M{1}^{+-\circ} +\M{0}^{\circ++} +\M{0}^{\circ+-}+\M{0}^{\circ-+} +\M{0}^{\circ--} +\M{1}^{-+\circ}
                              +\M{1}^{--\circ} .    \notag
\end{align}
\end{thm}

\begin{proof}
The formulae for $\TT{1}^{--\circ}$ and $\TT{1}^{-+\circ}$ are the same as \eqref{Tkd} in Theorem~\ref{thm:tilting:sing} with $kd=1$.

According to \cite[Lemma 3.3.2(iii)]{Ger00} we have the following character formula for the finite-dimensional irreducible $\g$-module $\mathbb L$ of highest weight $\delta+(p-1)\ep_1$:
\begin{align*}
{\rm ch }\; {\mathbb L} =\sum_{i=0}^{p-1} (e^{\delta+(p-1-2i)\ep_1}+e^{-\delta+(p-1-2i)\ep_1}) +\sum_{j=0}^p ( e^{(p-2j) \ep_1+\ep_2}+e^{(p-2j)\ep_1-\ep_2})
\end{align*}
We apply the translation functor $\mathcal E$ by tensoring $T_{0,-2,0}$ with $\mathbb L$ and then projecting to the block $\Bl_1$. This gives us
\begin{align}
 \label{ET0-20}
\mathcal ET_{0,-2,0}=\M{1}^{+-\circ} +\M{0}^{\circ-+} +\M{0}^{\circ--} +\M{1}^{--\circ}.
\end{align}
It follows from Lemmas \ref{lem:oddhom} and \ref{lem:evenhom} that all terms in \eqref{ET0-20}  must appear in a flag of $\TT{1}^{+-\circ}$, and hence $\mathcal ET_{0,-2,0} ={\TT{1}}^{+-\circ}$.

Similarly, in order to construct ${\TT{1}}^{++\circ}$ we apply $\mathcal E$ to $T_{0,2,0}=M_{0,2,0}+M_{0,-2,0}$. We have
\begin{align*}
\mathcal E T_{0,2,0}=\M{1}^{++\circ} + \M{1}^{+-\circ} +\M{0}^{\circ++} +\M{0}^{\circ+-}+\M{0}^{\circ-+} +\M{0}^{\circ--} +\M{1}^{-+\circ}+\M{1}^{--\circ} .
\end{align*}
Now, in addition to using Lemmas \ref{lem:oddhom} and \ref{lem:evenhom}, we make use of Lemma~\ref{tilting:comp} to establish that $\mathcal E T_{0,2,0}=\TT{1}^{++\circ} $.
\end{proof}

\begin{rem}
The Verma flags for tilting modules in $\Bl_1$ for $\ka =1/d$ with $d\in \Z_{\ge 2}$ can be read off
completely from Section \ref{subsec:tilting2} by using the Dynkin diagram symmetry.
We skip the detail except noting the  irregular tilting modules in the current setting are
$\TT{0}^{\circ\pm\pm}$,
$\TT{kd-1}^{-\pm\pm}$,
$\TT{1}^{+\pm\pm}$,
$\TT{-1}^{\pm\circ\pm}$,  $\TT{kd}^{+\pm\circ}$, $\TT{-2}^{+\pm\pm}$ and $\TT{kd+1}^{+\pm\pm}$.
\end{rem}

\subsection{Verma flags for projectives in $\Bl_1$} 
  \label{subsec:proj2}

The formulae of Verma flags for tilting modules in
Theorems~\ref{tilting:same}--\ref{thm:tilting:1:kd=1}
are readily translated into formulae for Verma flags of projective modules using the identity
\eqref{tilt=proj}. We formulate these results in Propositions~\ref{proj:same}--\ref{proj:1:kd=1} below.

Proposition~\ref{proj:reg} for regular projective modules remains valid here.
The irregular projective modules are:
$\PP{0}^{\circ\pm\pm}$,
$\PP{1-kp}^{+\pm\pm}$,
$\PP{-1}^{-\pm\pm}$,
$\PP{-kp}^{\pm\circ\pm}$,  $\PP{1}^{\pm\pm\circ}$, $\PP{-1-kp}^{-\pm\pm}$ and $\PP{2}^{-\pm\pm}$.

The earlier formulae for irregular projective modules $\PP{1-kp}^{+\pm\pm}$,
$\PP{-1}^{-\pm\pm}$, $\PP{-kp}^{\pm\circ\pm}$ and $\PP{-1-kp}^{-\pm\pm}$ remain valid.
More precisely, the formula \eqref{P1-kp} for $\PP{1-kp}^{+\pm\pm}$,
the formula \eqref{P1-1}  for $\PP{-1}^{-\pm\pm}$,
the formula \eqref{P-kp} for $\PP{-kp}^{\pm\circ\pm}$,
and the formula \eqref{P-1-kp} for $\PP{-1-kp}^{-\pm\pm}$
remain valid in the current setting. We summarize these as follows.

 \begin{prop}
   \label{proj:same}
Assume that  $\ka\in \Z_{\ge 2}$.
The formulae in Proposition~\ref{proj:reg} for regular projective modules
and the formulae for irregular projective modules $\PP{1-kp}^{+\pm\pm}$,
$\PP{-1}^{-\pm\pm}$, $\PP{-kp}^{\pm\circ\pm}$ and $\PP{-1-kp}^{-\pm\pm}$ in Section \ref{subsec:proj:rational} remain valid.
 \end{prop}

 The remaining new irregular cases
$\PP{0}^{\circ\pm\pm}$,  $\PP{2}^{-\pm\pm}$ and $\PP{1}^{\pm\pm\circ}$, respectively, are treated below.

\begin{prop}
  \label{proj:0:kd=1}  
Assume that $\ka\in \Z_{\ge 2}$.
We have the following Verma flags for the projective modules in the block $\Bl_1$:
\begin{align*}
\PP{0}^{\circ++}  & =\M{0}^{\circ++} +\M{-1}^{+++} +\M{1}^{++\circ} +\M{2}^{+++},    \notag
\\
\PP{0}^{\circ+-}  & =\M{0}^{\circ+-} +\M{0}^{\circ++} +\M{-1}^{++-} +\M{-1}^{+++}
                                +2 \M{1}^{++\circ}  +\M{2}^{++-} +\M{2}^{+++},    \notag
\\
\PP{0}^{\circ-+}  & =\M{0}^{\circ-+} +\M{0}^{\circ++} +\M{1}^{+-\circ} +\M{-1}^{+-+} +\M{-1}^{+++} +\M{1}^{++\circ}  +\M{2}^{+-+} +\M{2}^{+++},   
\\
\PP{0}^{\circ--}  & =\M{0}^{\circ--} +\M{0}^{\circ-+} +\M{0}^{\circ+-} +\M{0}^{\circ++}
+ 2\M{1}^{+-\circ} +\M{-1}^{+--} +\M{-1}^{+-+}   \notag \\
    &\qquad +\M{-1}^{++-}  +\M{-1}^{+++}  +2\M{1}^{++\circ}  +\M{2}^{+--} +\M{2}^{+-+} +\M{2}^{++-} +\M{2}^{+++}.    \notag
\end{align*}
\end{prop}

\begin{prop}
 \label{proj:2:kd=1}       
Assume that $\ka\in \Z_{\ge 2}.$
We have the following Verma flags for the projective modules in the block $\Bl_1$:
\begin{align}
\PP{2}^{-++}  &= \M{2}^{-++} +\M{1}^{-+\circ}  +\M{0}^{\circ++} +\M{1}^{++\circ}   + \M{2}^{+++},    \notag
\\
\PP{2}^{-+-}  &= \M{2}^{-+-} +\M{2}^{-++} +2\M{1}^{-+\circ}  +\M{0}^{\circ+-} +\M{0}^{\circ++} +2\M{1}^{++\circ}   + \M{2}^{++-} + \M{2}^{+++},    \notag
\\
\PP{2}^{--+}    &=  \M{2}^{--+} +\M{2}^{-++} +\M{1}^{--\circ} +\M{1}^{-+\circ}  +\M{0}^{\circ-+} +\M{0}^{\circ++}      \label{P2:kd=1}
                             \\ &\qquad +\M{1}^{+-\circ} +\M{1}^{++\circ}   + \M{2}^{+-+} + \M{2}^{+++},    \notag
\\
\PP{2}^{---}  &= \M{2}^{---} +\M{2}^{--+} +\M{2}^{-+-} +\M{2}^{-++} +2\M{1}^{--\circ}  +2\M{1}^{-+\circ}    +\M{0}^{\circ--} +\M{0}^{\circ-+}
\notag
\\
&+\M{0}^{\circ+-} +\M{0}^{\circ++}    + 2\M{1}^{+-\circ} + 2\M{1}^{++\circ}  + \M{2}^{+--} + \M{2}^{+-+} + \M{2}^{++-} + \M{2}^{+++}.   \notag
\end{align}
\end{prop}

\begin{prop}
   \label{proj:1:kd=1}             
Assume that $\ka\in \Z_{\ge 2}$. We have the following Verma flags for the projective modules in the block $\Bl_1$:
\begin{align}
\PP{1}^{++\circ}&=  \M{1}^{++\circ}  + \M{2}^{++-} + \M{2}^{+++},       \notag
\\
\PP{1}^{+-\circ}&= \M{1}^{+-\circ}  + \M{1}^{++\circ}  + \M{2}^{+--} + \M{2}^{+-+} + \M{2}^{++-} + \M{2}^{+++},   \label{P1:kd=1}
\\
\PP{1}^{-+\circ}  & =\M{1}^{-+\circ} +\M{0}^{\circ+-} +\M{0}^{\circ++} +\M{1}^{++\circ},    \notag
\\
\PP{1}^{--\circ}  & =\M{1}^{--\circ} + \M{1}^{-+\circ} +\M{0}^{\circ--} +\M{0}^{\circ-+}+\M{0}^{\circ+-} +\M{0}^{\circ++} +\M{1}^{+-\circ}
                              +\M{1}^{++\circ} .    \notag
\end{align}
\end{prop}

\subsection{Projective tilting modules in $\Bl_1$}  

As in Section \ref{subsec:prinj:kp,kd>1}, first by examining possible matches between the
Verma flags of tilting and projective modules and then using a similar argument as in Theorem \ref{T=P:irrational} we prove the following.

\begin{thm}
 \label{T=P:kd=1}
Assume $p\ge 2,  d=1$. We have the following isomorphisms between projective and tilting modules in $\Bl_1$:
\[
\TT{-p}^{+\circ+} \cong \PP{-p}^{-\circ-},
\quad
\TT{1}^{++\circ}\cong \PP{1}^{--\circ},
\quad
\TT{n}^{+++} \cong \PP{n}^{---} \;\; (n\in \Z \backslash \{0, 1, -p\}).
\]
Furthermore, there are no other projective tilting modules of atypical weights in $\Bl_1$.
\end{thm}

\subsection{Composition factors of Verma modules in $\Bl_1$} 
  \label{subsec:comp2}

Proposition~\ref{composition1:r} on composition factors for  Verma modules of regular highest weights remain valid here
(note now that the formulae for $\M{n}^{\pm\pm\pm}$ are valid for $n\ge 3$).

 We have the following irregular cases:
 $\M{0}^{\circ\pm\pm}, \M{1}^{\pm\pm\circ}, \M{2}^{\pm\pm\pm}, \M{-kp}^{\pm\circ\pm}, \M{1-kp}^{\pm\pm\pm}, \M{-1-kp}^{\pm\pm\pm}$.
It turns out that the composition factor formulae for $\M{1}^{\pm\pm\circ}$ and $\M{-kp}^{\pm\circ\pm}$ are the same as those in Proposition~\ref{composition:kd,kp>1}
 by setting $kd=1$,
and the composition factor formulae for $\M{2}^{\pm\pm\pm}$ is the same as those in Proposition~\ref{composition:+1} by setting $kd=1$ (where $\LL{kd-1}^{+**}$ is understood as $\LL{0}^{\circ **}$).
The composition factor formulae for $\M{1-kp}^{\pm\pm\pm}, \M{-1-kp}^{\pm\pm\pm}$ are the same as those given in Propositions~\ref{comp:1-kp:kp>1}, and \ref{comp-1-kp:kp>1}, respectively.
Finally, the composition factors for the remaining irregular Verma modules $\M{0}^{\circ\pm\pm}$ are given in Proposition~\ref{comp0:ka>1} below.

\begin{prop}
 \label{comp0:ka>1}
 Assume $\ka\in \Z_{\ge 2}$.
We have the following composition factors of Verma modules in $\Bl_1$:
\begin{align*}
\M{0}^{\circ--} =& \LL{0}^{\circ--}  +  \LL{1}^{--\circ} + \LL{-1}^{---} +  \LL{2}^{---},
 \\
\M{0}^{\circ-+} =& \LL{0}^{\circ-+} +\LL{0}^{\circ--} + \LL{1}^{--\circ}  + \LL{-1}^{--+}+ \LL{-1}^{---}  +\LL{2}^{--+} + \LL{2}^{---},
 \\
\M{0}^{\circ+-} =& \LL{0}^{\circ+-} +\LL{0}^{\circ--} +\LL{1}^{-+\circ} +\LL{1}^{--\circ} +\LL{-1}^{-+-} + \LL{-1}^{---} +\LL{2}^{-+-} + \LL{2}^{---},
 \\
\M{0}^{\circ++} =& \LL{0}^{\circ++} +\LL{0}^{\circ+-}  +\LL{0}^{\circ-+} +\LL{0}^{\circ--}
+ \LL{1}^{-+\circ}  + \LL{1}^{--\circ}
 \\
&\quad + \LL{-1}^{-++} + \LL{-1}^{-+-}  + \LL{-1}^{--+}  + \LL{-1}^{---}
+ \LL{2}^{-++} + \LL{2}^{-+-}  + \LL{2}^{--+}  + \LL{2}^{---}.
\end{align*}
\end{prop}

\section{Character formulae in the block $\Bl_1$, for $\ka=1$}
 \label{sec:block:ka=1}

We present the character formulae in the block $\Bl_1$ for the remaining case $\ka=1$, i.e., when $p=d=1$ in this section.
Recall the isomorphism $D(2|1;1) \cong \osp(4|2)$.
\subsection{Verma flags for tilting modules in $\Bl_1$ for $\ka=1$}
  \label{subsec:tilting1}

Theorem~\ref{thm:tilting:reg} for regular tilting modules remains valid in $\Bl_1$ for $\ka=1$ here.
The  irregular tilting modules in $\Bl_1$ with $\ka=1$ are
$\TT{0}^{\circ\pm\pm}$,
$\TT{-1}^{\pm\circ\pm}$,  $\TT{1}^{\pm\pm\circ}$, $\TT{-2}^{+\pm\pm}$ and $\TT{2}^{+\pm\pm}$.

\begin{thm}
  \label{tilting:Tpm2:p=d=1}
Assume that $\ka=1$.
The formulae for $\TT{2}^{+\pm\pm}$ are the same as in \eqref{T2:kd=1}.
The formulae for $\TT{-2}^{+\pm\pm}$ are given as follows:
\begin{align*}
\TT{-2}^{+--}  &= \M{-2}^{+--} +\M{-1}^{+\circ-}  +\M{0}^{\circ--} +\M{-1}^{-\circ-}   + \M{-2}^{---},    \notag
\\
\TT{-2}^{+-+}    &=  \M{-2}^{+-+} +\M{-2}^{+--} +\M{-1}^{+\circ+} +\M{-1}^{+\circ-}  +\M{0}^{\circ-+} +\M{0}^{\circ--}   
                             \\ &\qquad +\M{-1}^{-\circ+} +\M{-1}^{-\circ-}   + \M{-2}^{--+} + \M{-2}^{---},    \notag
\\
\TT{-2}^{++-}  &= \M{-2}^{++-} +\M{-2}^{+--} +2\M{-1}^{+\circ-}  +\M{0}^{\circ+-} +\M{0}^{\circ--} +2\M{-1}^{-\circ-}   + \M{-2}^{-+-} + \M{-2}^{---},    \notag
\\
\TT{-2}^{+++}  &= \M{-2}^{+++} +\M{-2}^{++-} +\M{-2}^{+-+} +\M{-2}^{+--}
+2\M{-1}^{+\circ+}  +2\M{-1}^{+\circ-}    +\M{0}^{\circ++} +\M{0}^{\circ+-}
\notag
\\
&+\M{0}^{\circ-+} +\M{0}^{\circ--}    + 2\M{-1}^{-\circ+} + 2\M{-1}^{-\circ-}
+ \M{-2}^{-++} + \M{-2}^{-+-} + \M{-2}^{--+} + \M{-2}^{---}.   \notag
\end{align*}
\end{thm}

\begin{proof}
Formulae for $\TT{-2}^{+\pm\pm}$ are in the same pattern as for $\TT{2}^{+\pm\pm}$, and they are obtained from \eqref{T2:kd=1} by symmetry.
\end{proof}

The proof of the following formulae is standard and will be skipped.
\begin{thm}
  \label{tilting:0:kappa=1}
Assume that $\ka=1.$
We have the following Verma flags for tilting modules in the block $\Bl_1$:
\begin{align}
\TT{0}^{\circ--}  & =\M{0}^{\circ--} +\M{-1}^{-\circ-} +\M{1}^{--\circ} +\M{-2}^{---} +\M{2}^{---},    \notag
\\
\TT{0}^{\circ-+}  & =\M{0}^{\circ-+} +\M{0}^{\circ--} +\M{-1}^{-\circ+} +\M{-1}^{-\circ-}
                                +2 \M{1}^{--\circ}  +\M{-2}^{--+} +\M{-2}^{---} +\M{2}^{--+} +\M{2}^{---},    \notag
\\
\TT{0}^{\circ+-}  & =\M{0}^{\circ+-} +\M{0}^{\circ--}  +\M{1}^{-+\circ} + 2\M{-1}^{-\circ-} +\M{1}^{--\circ}  +\M{2}^{-+-}+\M{-2}^{-+-} +\M{-2}^{---}  +\M{2}^{---},    \notag 
\\
\TT{0}^{\circ++}  & =\M{0}^{\circ++} +\M{0}^{\circ+-} +\M{0}^{\circ-+} +\M{0}^{\circ--}
+ 2\M{1}^{-+\circ} + 2\M{-1}^{-\circ+} + 2\M{-1}^{-\circ-} + 2\M{1}^{--\circ}  \notag \\
    &\qquad +\M{2}^{-++} +\M{2}^{-+-}  +\M{-2}^{-++}  +\M{-2}^{-+-}  +\M{-2}^{--+} +\M{-2}^{---} +\M{2}^{--+} +\M{2}^{---}.    \notag
\end{align}
\end{thm}

The formulae for $\TT{1}^{\pm\pm\circ}$ in Theorem \ref{thm:tilting:1:kd=1} remains valid for $\ka=1$.
By symmetry we have the following formulae for $\TT{-1}^{\pm\circ\pm}$.

\begin{thm}
  \label{tilting:1:kp=kd=1}
Assume that $\ka=1.$ The formulae for $\TT{1}^{\pm\pm\circ}$ are the same as in \eqref{T1:kd=1}.
The Verma flags for the tilting modules $\TT{-1}^{\pm\circ\pm}$ are given as follows:
\begin{align*}
\TT{-1}^{-\circ-}&=  \M{-1}^{-\circ-}  + \M{-2}^{-+-} + \M{-2}^{---},       \notag
\\
\TT{-1}^{-\circ+}&= \M{-1}^{-\circ+}  + \M{-1}^{-\circ-}  + \M{-2}^{-++} + \M{-2}^{-+-} + \M{-2}^{--+} + \M{-2}^{---},       
\\
\TT{-1}^{+\circ-} &= \M{-1}^{+\circ-} + \M{0}^{\circ+-} + \M{0}^{\circ--}+ \M{-1}^{-\circ-},
\notag
\\
\TT{-1}^{+\circ+} &=  \M{-1}^{+\circ+} +
\M{-1}^{+\circ-}+ \M{0}^{\circ++} + \M{0}^{\circ+-}  +\M{0}^{\circ-+} + \M{0}^{\circ--}+ \M{-1}^{-\circ+} + \M{-1}^{-\circ-}.
\notag
\end{align*}
\end{thm}

\subsection{Verma flags for projectives in $\Bl_1$ for $\ka=1$}
  \label{subsec:proj1}

The formulae of Verma flags for tilting modules in
Theorems~\ref{tilting:Tpm2:p=d=1}, \ref{tilting:0:kappa=1},  and \ref{tilting:1:kp=kd=1}
are readily translated into formulae of Verma flags for projective modules using the identity
\eqref{tilt=proj}. We formulate the results in Propositions~\ref{proj:Tpm2:p=d=1}, \ref{proj:0:kappa=1} and \ref{proj:1:kp=kd=1} below.

Theorem~\ref{proj:reg} for regular projective modules remains valid for $\ka=1$ here.
The  irregular tilting modules in $\Bl_1$ with $\ka=1$ are
$\PP{0}^{\circ\pm\pm}$,
$\PP{-1}^{\pm\circ\pm}$,  $\PP{1}^{\pm\pm\circ}$, $\PP{-2}^{-\pm\pm}$ and $\PP{2}^{-\pm\pm}$.

\begin{prop}
  \label{proj:Tpm2:p=d=1}
Assume that $\ka=1.$ Formulae for $\PP{2}^{-\pm\pm}$ in $\Bl_1$ are the same as in \eqref{P2:kd=1}.
Formulae for $\PP{-2}^{-\pm\pm}$ are given as follows:
\begin{align*}
\PP{-2}^{-++}  &= \M{-2}^{-++} +\M{-1}^{-\circ+}  +\M{0}^{\circ++} +\M{-1}^{+\circ+}   + \M{-2}^{+++},    \notag
\\
\PP{-2}^{-+-}    &=  \M{-2}^{-+-} +\M{-2}^{-++} +\M{-1}^{-\circ-} +\M{-1}^{-\circ+}  +\M{0}^{\circ+-} +\M{0}^{\circ++}   
                             \\ &\qquad +\M{-1}^{+\circ-} +\M{-1}^{+\circ+}   + \M{-2}^{++-} + \M{-2}^{+++},    \notag
\\
\PP{-2}^{--+}  &= \M{-2}^{--+} +\M{-2}^{-++} +2\M{-1}^{-\circ+}  +\M{0}^{\circ-+} +\M{0}^{\circ++} +2\M{-1}^{+\circ+}   + \M{-2}^{+-+} + \M{-2}^{+++},    \notag
\\
\PP{-2}^{---}  &= \M{-2}^{---} +\M{-2}^{--+} +\M{-2}^{-+-} +\M{-2}^{-++}
+2\M{-1}^{-\circ-}  +2\M{-1}^{-\circ+}    +\M{0}^{\circ--} +\M{0}^{\circ-+}
\notag
\\
&+\M{0}^{\circ+-} +\M{0}^{\circ++}    + 2\M{-1}^{+\circ-} + 2\M{-1}^{+\circ+}
+ \M{-2}^{+--} + \M{-2}^{+-+} + \M{-2}^{++-} + \M{-2}^{+++}.   \notag
\end{align*}
\end{prop}

\begin{prop}
  \label{proj:0:kappa=1}
Assume that $\ka=1.$
We have the following Verma flags for the projective modules in the block $\Bl_1$:
\begin{align*}
\PP{0}^{\circ++}  & =\M{0}^{\circ++} +\M{-1}^{+\circ+} +\M{1}^{++\circ} +\M{-2}^{+++} +\M{2}^{+++},    \notag
\\
\PP{0}^{\circ+-}  & =\M{0}^{\circ+-} +\M{0}^{\circ++} +\M{-1}^{+\circ-} +\M{-1}^{+\circ+}
                                +2 \M{1}^{++\circ}  +\M{-2}^{++-} +\M{-2}^{+++} +\M{2}^{++-} +\M{2}^{+++},    \notag
\\
\PP{0}^{\circ-+}  & =\M{0}^{\circ-+} +\M{0}^{\circ++}  +\M{1}^{+-\circ} + 2\M{-1}^{+\circ+} +\M{1}^{++\circ}  +\M{2}^{+-+}+\M{-2}^{+-+} +\M{-2}^{+++}  +\M{2}^{+++},    \notag
\\
\PP{0}^{\circ--}  & =\M{0}^{\circ--} +\M{0}^{\circ-+} +\M{0}^{\circ+-} +\M{0}^{\circ++}
+ 2\M{1}^{+-\circ} + 2\M{-1}^{+\circ-} + 2\M{-1}^{+\circ+} + 2\M{1}^{++\circ}  \notag \\
    &\qquad +\M{2}^{+--} +\M{2}^{+-+}  +\M{-2}^{+--}  +\M{-2}^{+-+}  +\M{-2}^{++-} +\M{-2}^{+++} +\M{2}^{++-} +\M{2}^{+++}.    \notag
\end{align*}
\end{prop}

\begin{prop}
  \label{proj:1:kp=kd=1}
Assume that $\ka=1.$ The formulae for $\PP{1}^{\pm\pm\circ}$ are the same as \eqref{P1:kd=1}.
The Verma flags for the tilting modules $\PP{-1}^{\pm\circ\pm}$ are given as follows:
\begin{align*}
\PP{-1}^{+\circ+}&=  \M{-1}^{+\circ+}  + \M{-2}^{+-+} + \M{-2}^{+++},       \notag
\\
\PP{-1}^{+\circ-}&= \M{-1}^{+\circ-}  + \M{-1}^{+\circ+}  + \M{-2}^{+--} + \M{-2}^{+-+} + \M{-2}^{++-} + \M{-2}^{+++},
\\
\PP{-1}^{-\circ+} &= \M{-1}^{-\circ+} + \M{0}^{\circ-+} + \M{0}^{\circ++}+ \M{-1}^{+\circ+},
\notag
\\
\PP{-1}^{-\circ-} &=  \M{-1}^{-\circ-} +
\M{-1}^{-\circ+}+ \M{0}^{\circ--} + \M{0}^{\circ-+}  +\M{0}^{\circ+-} + \M{0}^{\circ++}+ \M{-1}^{+\circ-} + \M{-1}^{+\circ+}.
\notag
\end{align*}
\end{prop}

\subsection{Projective tilting modules in $\Bl_1$ for $\ka=1$}

As in Section \ref{subsec:prinj:kp,kd>1}, first by examining possible matches between the
Verma flags of tilting and projective modules and then using a similar argument as in Theorem \ref{T=P:irrational} we obtain the following.\

\begin{thm}
 \label{T=P:p=d=1}
 Assume $\ka=1$. We have the following isomorphisms between projective and tilting modules in $\Bl_1$:
\[
\TT{-1}^{+\circ+} \cong \PP{-1}^{-\circ-},
\quad
\TT{1}^{++\circ}\cong \PP{1}^{--\circ},
\quad
\TT{n}^{+++} \cong \PP{n}^{---} \;\; (n\in \Z \backslash \{0, \pm 1\}).
\]
Furthermore, there are no other projective tilting modules of atypical weights in $\Bl_1$.
\end{thm}

\subsection{Composition factors of Verma modules in $\Bl_1$ for $\ka=1$}
  \label{subsec:comp1}

Proposition~\ref{composition1:r} on composition factors for Verma modules of regular highest weights (i.e., $\M{\pm n}^{\pm\pm\pm}$, for $n\ge 3$) remain valid here.

 We have the following irregular cases:
 $\M{0}^{\circ\pm\pm}, \M{1}^{\pm\pm\circ}, \M{-1}^{\pm\circ\pm}, \M{2}^{\pm\pm\pm}, \M{-2}^{\pm\pm\pm}$.

It turns out the composition factor formulae in $\M{1}^{\pm\pm\circ}$ and $\M{-1}^{\pm\circ\pm}$ are the same as those in Proposition~\ref{composition:kd,kp>1},
and the composition factor formulae in $\M{2}^{\pm\pm\pm}$ and $\M{-2}^{\pm\pm\pm}$ are the same as those in Propositions~\ref{composition:+1} and \ref{comp-1-kp:kp>1},
 by setting $kd=1$ and $kp=1$, respectively  (where $\LL{kd-1}^{+**}$ is understood as $\LL{0}^{\circ **}$ and  $\LL{1-kp}^{+**}$ is understood as $\LL{0}^{\circ **}$, respectively).

The composition factors in the remaining irregular Verma modules $\M{0}^{\circ\pm\pm}$ are given in Proposition~\ref{comp0:p=d=1} below.

\begin{prop}
 \label{comp0:p=d=1}
 Assume $\ka=1$.
We have the following composition factors of Verma modules in $\Bl_1$:
\begin{align*}
\M{0}^{\circ--} =& \LL{0}^{\circ--}  + \LL{1}^{--\circ} + \LL{2}^{---} + \LL{-1}^{-\circ-} + \LL{-2}^{---},
 \\
\M{0}^{\circ-+} =& \LL{0}^{\circ-+} +\LL{0}^{\circ--} + \LL{1}^{--\circ}  + \LL{-1}^{-\circ+}+ \LL{-1}^{-\circ-}
 + \LL{2}^{--+}+ \LL{2}^{---}  +\LL{-2}^{--+} + \LL{-2}^{---},
 \\
\M{0}^{\circ+-} =& \LL{0}^{\circ+-} +\LL{0}^{\circ--}
+\LL{1}^{-+\circ} +\LL{1}^{--\circ} +\LL{-1}^{-\circ-}
+\LL{2}^{-+-} + \LL{2}^{---}+\LL{-2}^{-+-} + \LL{-2}^{---},
 \\
\M{0}^{\circ++} =& \LL{0}^{\circ++} +\LL{0}^{\circ+-}  +\LL{0}^{\circ-+} +\LL{0}^{\circ--}
+ \LL{1}^{-+\circ}  + \LL{1}^{--\circ} + \LL{-1}^{-\circ+} + \LL{-1}^{-\circ-}
 \\
&\quad
+ \LL{2}^{-++} + \LL{2}^{-+-}  + \LL{2}^{--+}  + \LL{2}^{---}
+ \LL{-2}^{-++} + \LL{-2}^{-+-}  + \LL{-2}^{--+}  + \LL{-2}^{---}.
\end{align*}
\end{prop}

This completes our study of the character formulae in the BGG category $\OO$ of $\D$-modules of integral weights, for any parameter $\ka$.


\end{document}